\documentclass{amsart}
\usepackage{amsfonts,amssymb,latexsym}
\usepackage{amsmath,amsthm}
\usepackage{eucal}
\usepackage[latin1]{inputenc}
\usepackage{color}

\newtheorem{theorem}{Theorem}[section]
\newtheorem{lemma}[theorem]{Lemma}
\newtheorem{proposition}[theorem]{Proposition}

\theoremstyle{definition}
\newtheorem{definition}{Definition}[section]

\theoremstyle{remark}
\newtheorem{remark}{Remark}[section]
\def\R{{\mathbb R}}
\def\N{{\mathbb N}}

\def\H{{\mathcal H}}

\def\supp{\mathop{\rm supp}\nolimits}

\def\M{{\mathcal M}}
\newcommand{\dist}[2]{\Bigl\langle #1, #2 \Bigr\rangle}

\numberwithin{equation}{section}
\newcommand{\tendsto}[1]{\renewcommand{\arraystretch}{0.5}
\begin{array}[t]{c}
\longrightarrow \\
{ \scriptstyle #1 }
\end{array}
\renewcommand{\arraystretch}{1}}

\newcommand{\weaktendsto}[1]{\renewcommand{\arraystretch}{0.5}
\begin{array}[t]{c}
\rightharpoonup \\
{ \scriptstyle #1 }
\end{array}
\renewcommand{\arraystretch}{1}}

\begin{document}
\title[Asymptotic stability of the Degasperis-Procesi peakons ]{A rigidity result for the Holm-Staley $ b$-family of  equations with application to the asymptotic stability of the Degasperis-Procesi peakon }


\author[L. Molinet]{Luc Molinet}


\address{Luc Molinet, Institut Denis Poisson, Universit\'e de Tours, Universit\'e d'Orl\'eans, CNRS,  Parc Grandmont, 37200 Tours, France.}
\email{Luc.Molinet@univ-tours.fr}

\date{\today}
\begin{abstract}
We prove  that the peakons are  asymptotically  $H^1$-stable, under the flow of the Degasperis-Procesi equation, in the class of  functions with a momentum density that belongs to $ {\mathcal M}_+(\R) $.
The key argument is a  rigidity result for uniformly in time exponentially decaying global solutions that is shared by  the 
Holm-Staley $ b$-family of equations for $b\ge 1$.
This extends previous results obtained for the Camassa-Holm equation ($b=2$). 
 \end{abstract}

\maketitle

\section{Introduction}
\noindent The  Degasperis-Procesi equation (D-P) reads 
\begin{equation}
u_t -u_{txx}=- 4 u u_x +3u_x u_{xx} + u u_{xxx},\quad 
(t,x)\in\R^2, \label{DP1}\\
\end{equation}
It is a particular case of the more general family of equations called $ b$-family equation that reads
\begin{equation}
u_t -u_{txx}=- (b+1) u u_x +bu_x u_{xx} + u u_{xxx}, 
(t,x)\in\R^2, \label{bfamily1}\\
\end{equation}
with $ b\in \R $ . This family of equations was introduced by Holm and Staley (\cite{HS1},\cite{HS2}) to  study the exchange of stability in the dynamics of solitary wave solutions under changes in the nonlinear balance in a  one dimensional shallow water waves equation. It reduces for $ b=3 $ to the DP equation and for $ b=2 $ to the famous
 Camassa-Holm equation (C-H) 
 \begin{equation}
u_t -u_{txx}=- 3 u u_x + 2 u_x u_{xx} + u u_{xxx}, 
(t,x)\in\R^2\, . \label{Camassa1}\\
\end{equation}
Even if, as noticed by Holm and Staley,   the behavior of the solutions  do change at $ b=0,\mp 1, \mp 2, \mp 3 $, the $b$-family equations share the same peaked solitary waves given by 
$$ u(t,x)=\varphi_c(x-ct)=c\varphi(x-ct)=ce^{-|x-ct|},\; c\in\R.$$
They are called peakon whenever $ c>0 $ and antipeakon  whenever
$c<0$.  
Note that the initial value problem associated with \eqref{bfamily1} has to be rewriten as
\begin{equation}
\left\{ \begin{array}{l}
u_t +u u_x +(1-\partial_x^2)^{-1}\partial_x (\frac{b}{2}u^2+\frac{3-b}{2}u_x^2)=0\\
\label{bfamily} \; 
u(0)=u_0, 
\end{array}
\right.
\end{equation}
to give a meaning to these solutions. 
It is also worth noticing that the $b$-family equation \eqref{bfamily1} can be rewritted as
\begin{equation}\label{by}
y_t +u y_x+b u_x y =0 
\end{equation}
 which is a transport equation for the momentum density  $y=u-u_{xx} $.
  As a consequence an initial data with a signed initial momentum density gives rise to a solution that keeps this property.
  This is one of the main point to prove that  smooth initial data with integrable  signed initial momentum density give rise to global solutions (see \cite{GLT}).

Both the Camassa and the Degasperis-Procesi equation 
can be derived as a model for the propagation of unidirectional
shallow water waves over a flat bottom  (\cite{CH1}, \cite{Johnson}), \cite{AL} and  \cite{CL}.    
They are also known to be  completely integrable (see \cite{CH1},\cite{CH2}, \cite{DP1}, \cite{DP2}) and to be bi-Hamiltonian.  
They both can be written in Hamiltonian form as 
\begin{equation}
\partial_t E'(u) =-\partial_x F'(u) \quad .
\end{equation}
For the Camassa-Holm equation it holds 
\begin{equation} 
\; E(u)=\int_{\R} u^2+u^2_x
\,  \mbox{ and } F(u)=\int_{\R} u^3+u u^2_x \,\;
\label{Ech}
\end{equation}
whereas for the Degasperis-Procesi equation it holds 
\begin{equation}\label{Edp}
E(u)=\H(u)=\int_{\mathbb{R}}yv=\int_{\R} 5v^2 +4 v_x^2 +v_{xx}^2 ~~\text{and}~~F(u)=\int_{\mathbb{R}}u^{3}
\end{equation}
with $v=(4-\partial^{2}_{x})^{-1}u $.

It is worth noticing that they are the only elements of the $ b$-family that  enjoy such Hamiltonian structure 
   and thus for which results on the orbital stability of the peakon do exist\footnote{However, numerical simulations (\cite{HS1})  seems to indicate that the peakons are stable under the flow of \eqref{bfamily} as soon as $ b>1 $}. In a pioneer work  \cite{CS1}, Constantin and Strauss  proved the orbital stability in $ H^1(\R) $ of the peakon for the Camassa-Holm equation. This approach has been then adapted in \cite{LL} for the Degasperis-Procesi equation but only in the case of a non negative density momentum. Note also that a great simplification of the proof has been done in \cite{Andre2}. However, as far as the author knows, there is no available orbital stability  result for the peakon of the  DP equation with respect to solutions with 
non signed  momentum density. 
 
 In \cite{L}, the author proved that  the peakons are asymptotically stable under the CH-flow
  in the class of solutions emanating from initial data that belong to $ H^1(\R) $ with a density momentum that is a non negative finite measure. The main point was the proof of a rigidity result for uniformly almost localized solutions of the CH equation in this class. The asymptotic stability result being then obtained by following the approach developed by Martel and Merle (\cite{MM1}, \cite{MM2}). 
  
In this paper we first prove that this rigidity property can be extended to the $ b$-family equations whenever $ b\ge 1 $. This emphasizes that this property is not directly linked to the integrability structure of the equation since it was proven in \cite{DP1} that \eqref{bfamily} is not integrable for $ b\not\in\{2,3\} $.
In a second part, we establish the asymptotic stability of the peakons under the DP-flow in  the class of solutions emanating from initial data that belong to $ H^1(\R) $ with a density momentum that is a non negative finite measure.

Before stating our results let us introduce the function space where our initial data will take place. Following \cite{CM}, we introduce the following space of functions
\begin{equation}
Y=\{ u\in H^1(\R)    \mbox{ such that  }  u-u_{xx} \in {\mathcal M}(\R)  \} \; .
\end{equation}
We denote by $ Y_+ $ the closed subset of $ Y $ defined by  $Y_+=\{u\in Y \, /\, u-u_{xx}\in \M_+ \} $  where $ {\mathcal M}_+ $ is the set of non negative finite Radon measures on $ \R$.
Note that since we are not aware of available uniqueness, global existence and continuity with respect to initial data results of all the $ b$-family equations for initial data in this class, we establish such results in Section \ref{section2}.

Let $ C_b(\R)$  be the set of bounded continuous functions on $ \R $, $ C_0(\R)$  be  the set of continuous functions  on $ \R $ that tends to $ 0 $ at infinity and let $ I \subset \R$ be an interval.  
 A sequence $\{\nu_n\}\subset {\mathcal M} $ is said to converge tightly (resp. weakly) towards $ \nu\in {\mathcal M} $ if for any $ \phi\in C_b(\R) $ (resp. $C_0(\R)$), $ \langle \nu_n,\phi\rangle \to
  \langle \nu,\phi\rangle $. We will then write $ \nu_n  \rightharpoonup \! \ast \; \nu $ tightly  in $ \M $ (resp. $ \nu_n  \rightharpoonup \! \ast \; \nu $ in $\M$).
 
 Throughout this paper, $ y\in C_{ti}(I;\M) $ (resp.   $ y\in C_{w}(I;\M) $) will signify that for any $ \phi\in C_b(\R) $
  (resp. $\phi\in C_0(\R)$) , 
 $ t\mapsto \dist{y(t)}{\phi} $ is continuous on $ I$ and $ y_n \rightharpoonup \! \ast \; y $ in $ C_{ti}(I;\M) $ (resp. $ y_n \rightharpoonup \! \ast \; y $ in $ C_{w}(I;\M) $) will signify that for any $ \phi\in C_b(\R) $ (resp. $C_0(\R)$),
 $ \dist{y_n(\cdot)}{\phi}\to \dist{y(\cdot)}{\phi} $ in $ C(I)$. 
\begin{definition}\label{defYlocalized}
 We say that  a solution $ u \in C(\R; H^1(\R)) $ with $ u-u_{xx}\in C_{w}(\R; \M_+) $ of \eqref{bfamily} is $ Y$-almost localized if there exist $ c>0 $ and a $ C^1 $-function $ x(\cdot) $, with $ x_t\ge c>0 $,  for which for any $ \varepsilon>0 $, there exists $ R_{\varepsilon},>0 $ such that for all $ t\in\R $ and all $ \Phi\in C(\R) $ with $0\le \Phi\le 1 $ and $ \supp \Phi \subset [-R_\varepsilon,R_\varepsilon]^c $.
  \begin{equation}\label{defloc}
  \quad \Bigl\langle  \Phi(\cdot-x(t)), u(t)-u_{xx}(t)\Bigr\rangle \le \varepsilon \; .
\end{equation}
We will say that such solution is uniformly in time exponentially decaying (up to translation)  if there exist 
$ a_1, a_2>0 $ such that for all $(t ,x)\in \R ^2$ 
\begin{equation}\label{exp}
|u(t,x)| \le a_1  e^{-a_2 |x-x(t)|} 
\end{equation}
 \end{definition}
\begin{remark}\label{remark1}
In \cite{L} we replaced  \eqref{defloc} by 
\begin{equation}\label{defancien}
 \int_{\R} (u^2(t)+u_x^2(t))  \Phi(\cdot-x(t)) \, dx + \Bigl\langle  \Phi(\cdot-x(t)), u(t)-u_{xx}(t)\Bigr\rangle \le \varepsilon \; .
\end{equation}
 in the definition of $ Y $-almost localized solutions.
This characterization was natural in the context fo the Camassa-Holm equation since the $ H^1 $-norm is a conservation law. For the $ b$-family this is not the case and this is  why it seems more appropriate to give the characterization \eqref{defloc} that is however equivalent to \eqref{defancien} (see the beginning of Section 
\ref{section2} for the proof of this equivalence).
\end{remark}
\begin{theorem}\label{rigidity}
 Let $ b\ge 1 $ and  $ u \in C(\R; H^1(\R)) $,  with $ u-u_{xx}\in C_{w}(\R; \M_+) $, be a $ Y$-almost localized solution of 
  the $ b$-family equation \eqref{bfamily}  that is not identically vanishing. 
  Assume moreover that 
  $ u $ is uniformly in time exponentially decaying  in the case $ b\not\in\{1,2,3\}$. 
   Then there exists $ c^*	>0 $ and $ x_0\in \R $ such that  $$
 u(t)=c^* \, \varphi(\cdot -x_0-c^* t) , \quad \forall t\in \R \, .
 $$
\end{theorem}

As a consequence we get the asymptotic stablity of the peakons  to the  DP equation ($b=3$)  that reads in the weak
 form as 
 \begin{equation}\label{DP}
 u_t + u u_x + \frac{3}{2} \partial_x (1-\partial_x^2)^{-1} u^2 =0  \; .
 \end{equation}
\begin{theorem} \label{asympstab} 
Let $ c>0 $ be fixed. There exists a constant $0<\eta\ll 1  $ such that for any $ 0<\theta<1 $ and any $ u_0\in Y_+ $ satisfying 
\begin{equation}\label{difini}
\H(u_0-\varphi_c) \le \eta \, \theta^8 \; ,
\end{equation}
there exists $ c^*>0 $ with $ |c-c^*|\ll c $ and a $C^1$-function $ x \, : \, \R \to \R $ 
 with $ \displaystyle\lim_{t\to \infty} \dot{x}=c^* $  such that
\begin{equation}\label{cvfaible}
u(t,\cdot+x(t)) \weaktendsto{t\to +\infty} \varphi_{c^*} \mbox{ in } H^1(\R)  \; ,
\end{equation}
where $ u\in C(\R; H^1) $ is the solution of the DP equation ($b=3$) emanating from $ u_0 $.
Moreover, for any $z\in \R $, 
\begin{equation}\label{cvforte}
\lim_{t\to +\infty} \|u(t) -\varphi_{c^*}(\cdot-x(t))\|_{H^1(]-\infty,z[\cup ]\theta t ,+\infty[)}=0 \; .
\end{equation}
\end{theorem}
\begin{remark} Using that \eqref{DP} is invariant by the change of unknown $ u(t,x)\mapsto -u(t,-x) $, we obtain as well the asymptotic stability of the antipeakon profile $ c \varphi $ with $ c<0 $ in the class of $ H^1$-function with a momentum density that belongs to $ \M_-(\R) $.
\end{remark}
\begin{remark} We choose to write the convergence results \eqref{cvfaible}-\eqref{cvforte} in $H^1 (\R) $ by sake of simplicity. Actually these convergence results hold in any $ H^s(\R) $ for $ 0\le s<3/2 $. 
\end{remark}

This paper is organized as follows : the next section is devoted to the proof of  the well-posedness results of the $ b$-family equations in the class of solutions we will work with. 
In Section \ref{sect3}, we prove the rigidity result  for $ Y$-almost localized global solutions of the $ b $-family equations with $ b\ge 1 $ with the additional uniform exponential decay condition as soon as $ b>1$.  Finally, in Section \ref{sect4} we prove Theorem \ref{asympstab} as well as an  asymptotic stability  result for  train of peakons to the DP equation. Note that we postponed to the appendix  the proof of an almost monotonicity result for the DP equation that leads to the uniform exponential decay result for $ Y$-almost localized global solutions.

\section{Global well-posedness results} \label{section2}
We first recall some obvious estimates that will be useful in the sequel of this paper.  
Noticing that $ p(x)=\frac{1}{2}e^{-|x|} $ satisfies $p\ast y=(1-\partial_x^2)^{-1} y $ for any $y\in H^{-1}(\R) $ we easily get
$$
\|u\|_{W^{1,1}}=\|p\ast (u-u_{xx}) \|_{W^{1,1}}\lesssim \| u-u_{xx}\|_{\M}
$$
and 
$$
\|u_{xx}\|_{\M}\le \|u\|_{L^1}+ \|u-u_{xx} \|_{\M} 
$$
 which ensures  that 
\begin{equation} \label{bv}
Y\hookrightarrow  \{ u\in W^{1,1}(\R) \mbox{ with } u_x\in {\mathcal BV}(\R) \} \; .
\end{equation}
Moreover,  Young's convolution inequalities lead to 
\begin{equation}\label{ut}
\max(\|u\|_{L^2},\|u\|_{L^\infty}, \|u_x\|_{L^2} , \|u_x\|_{L^\infty})\le \| u-u_{xx}\|_{\M} \; .
\end{equation}
 It is also worth noticing that since for $ u\in C^\infty_0(\R) $, 
$$
u(x)=\frac{1}{2} \int_{-\infty}^x e^{x'-x} (u-u_{xx})(x') dx' +\frac{1}{2} \int_x^{+\infty} e^{x-x'} (u-u_{xx})(x') dx'
$$
and 
$$
u_x(x)=-\frac{1}{2}\int_{-\infty}^x e^{x'-x} (u-u_{xx})(x') dx' + \frac{1}{2}\int_x^{+\infty} e^{x-x'} (u-u_{xx})(x') dx' \; ,
$$
we get $ u_x^2 \le u^2 $ as soon as $ u-u_{xx} \ge 0 $ on $ \R $. By the density of $ C^\infty_0(\R) $ in $ Y $, we deduce that 
\begin{equation}\label{dodo}
 |u_x|\le u\mbox { for any } u\in Y_+ \; .
\end{equation}
In this paper we will often use that the $ Y$-almost localization of a global solution $ u $ leads to a $L^p $-almost localization of $u $ for $ p\in [1,+\infty] $ and even to a $ H^1$-almost localization.  Indeed, for $ u$  satisfying Definition \ref{defYlocalized}, taking $ \tilde{\Phi} \in C(\R) $ with 
$ 0\le \tilde{\Phi} \le 1 $,  $\tilde{\Phi}\equiv 0 $ on $[-R_{\varepsilon/8},R_{\varepsilon/8}] $
 and $\tilde{\Phi}\equiv 1 $ on $[-2R_{\varepsilon/8},2R_{\varepsilon/8}]^c $, we get  for $ p\in [1,+\infty] $ and any $ R>2R_{\varepsilon/8} $ that 
$$
\|u(t,\cdot+x(t))\|_{L^p(|x|>R)} \le \frac{1}{2} \| e^{-|\cdot|} \ast  (y(t,\cdot+x(t)) \tilde{\Phi})\|_{L^p(|x|>R)} 
+\frac{1}{2} \Bigl\| e^{-|\cdot|} \ast ( y(t,\cdot+x(t)) (1-\tilde{\Phi})) \Bigr\|_{L^p(|x|>R)} \; .
$$
with 
$$
 \| e^{-|\cdot|} \ast  (y(t,\cdot+x(t)) \tilde{\Phi})\|_{L^p(|x|>R)}\le \|e^{-|\cdot|} \|_{L^p} \|y(t,\cdot+x(t))  \tilde{\Phi}\|_{\M} \le  \varepsilon/4 \; 
 $$
 and for $ p\neq +\infty $, 
 \begin{align*}
  \Bigl\| e^{-|\cdot|} \ast \Bigl( y (t,\cdot+x(t))(1-\tilde{\Phi})\Bigr) \Bigr\|_{L^p(|x|>R)} & =
   \Bigl(  \int_{|x|>R} \Bigl| \langle e^{-|x-\cdot|} (1-\tilde{\Phi}), y(t,\cdot+x(t)) \rangle \Bigr|^p  \Bigr)^{1/p}\\
   & \le e^{2R_{\varepsilon/8}} M(u) \|e^{-|\cdot|} \|_{L^p(|x|>R)} \le 2 e^{-R+2R_{\varepsilon/8}} M(u)\; .
 \end{align*}
 with obvious modifications for $ p=+\infty $. Taking $ R>2R_{\varepsilon/8} $ such that $ 2e^{-R+2R_{\varepsilon/8}}M(u)<\varepsilon/2$, this ensures that 
 $$
 \|u(\cdot+x(t)\|_{L^p(|x|>R)} \le \varepsilon
 $$
 and applying this estimate for $ p=2$ with \eqref{dodo} 
  in hands, we infer that \eqref{defancien} holds for $ R_\varepsilon' > 2 R_{\varepsilon/8} $ with 
  $ e^{- R_\varepsilon'}M(u) < R_{\varepsilon/8}  \,  \varepsilon/8 $. 

Finally, throughout this paper, we will denote $ \{\rho_n\}_{n\ge 1} $ the mollifiers defined by 
\begin{equation} \label{rho}
\rho_n=\Bigl(\int_{\R} \rho(\xi) \, d\xi 
\Bigr)^{-1} n \rho(n\cdot ) \mbox{ with } \rho(x)=\left\{ 
 \begin{array}{lcl} e^{1/(x^2-1)} & \mbox{for} & |x|<1 \\
0 & \mbox{for} & |x|\ge 1 
\end{array}
\right.
\end{equation}
In \cite{GLT}  the global well-posedness result for smooth solutions with a non negative momentum density of the Camassa-Holm equation (see \cite{CE1})  is adapted to \eqref{bfamily}. This result can be summarized in the following proposition 
 \begin{proposition}{(Strong solutions \cite{GLT})}\label{smoothWP} \\
 Let $ u_0\in H^s(\R)$ with $ s>3/2 $. Then the initial value problem \eqref{bfamily} has a  unique solution $ u\in C([0,T]; H^s(\R)) 
 \cap C^1([0,T]; H^{s-1}) $ where $ T=T(\|u_0\|_{H^{\frac{3}{2}+}})>0 $ and, for any $ r>0 $,  the map $ u_0 \to u $ is continuous  from $ B(0,r)_{H^s} $ into $ C([0,T(r); H^s(\R)) $\\
 Moreover, let $ T^* >0 $ be the maximal time of existence of $ u $ in $ H^s(\R)$ then 
 \begin{equation}
 T^*<+\infty \quad \Leftrightarrow  \quad \liminf_{t\nearrow T} u_x=-\infty 
 \end{equation}
 and if $y_0= u_{0}-u_{0,xx}\ge 0 $ with $ y_0\in L^1(\R) $ then $ T^*=+\infty $ and 
 \begin{equation}\label{estimatey}
   \|y(t)\|_{L^1} = \|y(0)\|_{L^1} , \forall t\in \R_+ \; .
  \end{equation}
\end{proposition}

Unfortunately, the peakons do not enter in this framework since their profiles do not belong  even to $ H^{\frac 3 2}(\R) $. In \cite{CM} an existence and uniqueness result of global solutions to Camassa-Holm in  a class of functions that contains the peakon is proved. This result was shown to hold also for the DP equation in \cite{LY}. We check below that it can be  actually extended to the whole $ b$-family and thus do not require  the hamiltonian structure of the equation.  
 \begin{proposition}\label{WP} 
 Let $ u_0 \in Y_+  $ be given. \vspace*{2mm} \\
 {\bf 1. Existence and uniqueness :}  Then the  IVP  \eqref{bfamily} has a global solution $ u\in C^1(\R;L^1(\R)\cap L^2(\R))\cap C(\R;W^{1,1}\cap H^1(\R)) $ such that 
  $ y=(1-\partial_x^2)u \in C_{w}(\R; \M)$ 
  Moreover, this solution is unique in the class 
   \begin{equation}\label{clas}
  \{ f\in C(\R_+; H^1(\R)) \cap \{f-f_{xx} \in L^\infty(\R+; \M_+) \} \; .
  \end{equation}
   \vspace*{2mm} \\
  {\bf 2. Continuity with respect to initial data  in $ W^{1,1}(\R)$}: For any sequence $ \{u_{0,n}\} $ bounded in $ Y_+ $ such that $ u_{0,n} \to u_0 $ in $ W^{1,1}(\R ) $,  the emanating sequence of solution $ \{u_n\} \subset  C^1(\R;L^1(\R)\cap L^2(\R))\cap C(\R;W^{1,1}\cap H^1(\R))  $ satisfies for any $ T>0 $
  \begin{equation}\label{cont1}
  u_n \to u \mbox{ in } C([-T,T]; W^{1,1}\cap H^1(\R))  \end{equation}
  and 
    \begin{equation}\label{cont2}
 (1-\partial_x^2) u_n  \rightharpoonup  \! \ast \; y \mbox{ in } C_{ti} ([-T,T], \M) \; . 
   \end{equation}
       {\bf 3. Continuity with respect to initial data  in $Y$ equipped with its weak topology} : For any sequence $ \{u_{0,n}\} \subset Y_+ $ such that\footnote{By this we mean that $ u_{0,n} \rightharpoonup u_0 $ in $ H^1(\R) $ and $ u_{0,n} \rightharpoonup \! \ast \; u_0 $ in $ \M$} $ u_{0,n} \rightharpoonup \! \ast \; u_0 $ in $ Y $,  the emanating sequence of solution $ \{u_n\} \subset C^1(\R;L^1(\R)\cap L^2(\R))\cap C(\R;W^{1,1}\cap H^1(\R)) $ satisfies for any $ T>0 $,
  \begin{equation}\label{weakcont}
  u_n \weaktendsto{n\to\infty} u \mbox{ in } C_{w}([-T,T]; H^1(\R) ) \; ,
  \end{equation}
  and \eqref{cont2}.
  \end{proposition}
   \begin{proof}
  {\bf 1.}
  First note that in \cite{GLT} the results  are stated only for positive times but, since the equation is invariant by the change of unknown $u(t,x) \mapsto u(-t,-x) $, it is direct to check that the results hold as well for negative times. 

  {\it Uniqueness} The uniqueness follows the same lines as for the Camassa-Holm equation (see \cite{CM})
  Let $ u $ and $ v$ be two solutions of \eqref{bfamily} within the class \eqref{clas}.
  At this stage it is worth noticing that \eqref{bfamily} then  ensures that $u,v\in C^1(\R;L^1(\R)\cap L^2(\R)) $. We set $ w=u-v $ and 
  $$
  M=\sup_{t\ge 0} \| u(t,\cdot)-u_{xx}(t,\cdot) \|_{\M} + \| v(t,\cdot)-v_{xx}(t,\cdot)\|_{\M} \; .
  $$
  One can easily check (see \cite{CM}) that 
  $$
  \sup_{(t,x)\in \R^2}(|u(t,x)|, |v(t,x)| , |u_x(t,x)|, |v_x(t,x)|)\le \frac{1}{2} M 
  $$
  and 
  $$
  \sup_{t\in\R}(\|u(t)\|_{L^1}, \|v(t)\|_{L^1}, \|u_x(t)\|_{L^1}, \|v_x(t)\|_{L^1} ) \le M \; .$$
  We proceed exactly as in \cite{CM}, using exterior regularization. It holds 
  \begin{eqnarray*}
  \frac{d}{dt} \int_{\R} |\rho_n \ast w | &= & \int_{\R} (\rho_n \ast w_t) \, \text{sgn}\, (\rho_n \ast w) \\
  & \le & \frac{1+|b|}{2} M  \int_{\R} |\rho_n \ast w \| +\frac{1+|3-b|}{2} M  \int_{\R} |\rho_n \ast w_x \|  +R_n(t)
  \end{eqnarray*}
  and 
  \begin{eqnarray*}
  \frac{d}{dt} \int_{\R} |\rho_n \ast w_x \| &= & \int_{\R} (\rho_n \ast w_{tx})  \, \text{sgn}\,(\rho_n \ast w_x) \\
  & \le & \frac{4+|b|}{2} M  \int_{\R} |\rho_n \ast w \| +\frac{5+|3-b|}{2} M  \int_{\R} |\rho_n \ast w_x \|  +R_n(t)
  \end{eqnarray*}
  with 
  $$
  R_n(t) \to 0 \text{ as } n\to +\infty  \; \text{ and } |R_n(t)| \lesssim 1, \; n\ge 1, t\in\R.
  $$
  Gathering these two inequalities and using Gronwall lemma, we thus get 
  $$
  \int_{\R} ( |\rho_n \ast w |+ |\rho_n \ast w_x |)(t,x)\le \int_0^t e^{c(b)(t-s)} R_n(s) ds 
  + \Bigl[ \int_{\R}  ( |\rho_n \ast w |+ |\rho_n \ast w_x |)(0,x)\Bigr] e^{c(b)t}
  $$
  with $ c(b)=  $. Fixing $ t\in\R $ and letting $ n\to +\infty $ this leads to 
  \begin{equation}\label{estimw11}
  \|w(t)\|_{W^{1,1}} \le e^{c(b)t} \|w_0\|_{W^{1,1}}, \quad \forall t\in \R \; .
  \end{equation}
  This yields the uniqueness in the class \eqref{clas}.  \\
  {\it Existence :} Let $ u_0 \in Y_+$.  According to Proposition \ref{smoothWP} for any $ n\ge 1$, $ \rho_n\ast u_0 $ gives rise to a global smooth solution $ u^n $ of \eqref{bfamily} and \eqref{estimatey} ensures that $(u^n)_{n\ge 1} $ is bounded in 
   $ C(\R; Y) $. As $ \rho_n \ast u_0 \to u_0 $ in $ W^{1,1}$, fixing $ T>0 $,  we deduce from \eqref{estimw11} that $(u^n)_{n\ge 1} $ is a Cauchy sequence in $ C([-T,T] ; W^{1,1}) ) $ and thus in  $ C([-T,T] ; H^1(\R) ) $ in view of \eqref{dodo}. 
   By a diagonal process we thus construct a function $u\in C(\R; W^{1,1}\cap H^1) $ that satisfies \eqref{bfamily} in $ L^2(\R) $ for all $ t\in \R $. Moreover, \eqref{estimatey}  ensures that $ u$ belongs to the uniqueness class \eqref{clas}
    and thus to $ C^1(\R; L^1(\R) \cap L^2(\R))$ as noticed in the uniquess proof.
  It remains to prove   that 
  \begin{equation}\label{tr}
 (1-\partial_x^2) u^n  \rightharpoonup  \! \ast \; y =u-u_{xx} \mbox{ in } C_{ti} ([-T,T], \M)   \;.
 \end{equation}
 Indeed, this will force $ M(u) $ to be a conserved quantity  since $ M(u_n) $ is a conserved quantity.  
 To do this, we notice that for any $ v\in BV(\R) $ and any $ \phi\in C^1_b(\R) $, it holds 
  $$
  \langle v', \phi\rangle = -\int  v \phi' \; .
  $$
  Therefore, setting $ y^n=u^n-u^n_{xx} $, The convergence of $ u^n $ towards $ u $ in $ C([-T,T] ; W^{1,1}) )$, $T>0 $,  ensures that, for any $ t\in \R $,
  $$
  \int_{\R} y^n(t) \phi =\int_{\R} (u^n (t)\phi+u^n_{x} (t) \phi') \tendsto{n\to+\infty} \int_{\R}  (u(t) \phi+u_{x}(t)  \phi')
  = \langle y(t), \phi \rangle 
  $$
  and thus  $  y^n(t) \rightharpoonup  \! \ast \; y(t) $ tightly in $\M$. Using that the equation \eqref{bfamily} forces $ \{\partial_t u^{n}\} $ to be  bounded in 
   $ L^\infty(0,T;L^1) $, Arzela-Ascoli theorem leads then  to \eqref{tr}.  Indeed, for any $ \phi\in C^2_b(\R) $ and any $ T>0 $,  we observe that the sequence of $C^1 $  functions $\{ t\mapsto \langle y^n(t),\phi\rangle \}$ is uniformly equi-continuous on $ [0,T] $ since 
    $$
    |\frac{d}{dt}   \langle y^n(t),\phi\rangle|=|\int_{\R} u^n_t (\phi-\phi_{xx})| \le 2 \|u^n_t \|_{L^\infty(0,T;L^1)} \|\phi\|_{C^2}  \; .
    $$
  {\bf 2.} It is clear that \eqref{cont1} is a direct consequence of the $ W^{1,1}$-Lipschitz bound \eqref{estimw11} and 
  \eqref{cont2} follows then exactly as above.\\
  {\bf 3.} According to Banach-Steinhaus Theorem, $\{u_{0,n}\} $ is bounded in $ Y_+$. Therefore, the sequence of emanating solution $\{u_n\} $ is bounded in  $ C(\R_+;W^{1,1}\cap H^1(\R)) $
  with $ \{u_{n,x} \} $ bounded in $ L^\infty(\R; {\mathcal BV}(\R)) $.  Hence, there exists $ v\in  L^\infty(\R_+;H^1(\R)) $ with $ (1-\partial_x^2) v \in 
   L^\infty(\R; {\mathcal M_+(\R)}) $ such that, for any $ T>0$,
   $$
     u_n \weaktendsto{n\to\infty} v \in L^\infty([-T,T]; H^1(\R))  \mbox{ and }  (1-\partial_x^2) u_n \weaktendsto{n\to\infty} \hspace*{-3mm} \ast \; (1-\partial_x^2) v  \mbox{ in } L^\infty(]-T,T[; {\mathcal M}_+(\R))  
   $$
   But, using that $ \{\partial_t u_n\} $ is bounded in $L^\infty(\R; L^2(\R) \cap L^1(\R) )$, Helly's,  Aubin-Lions compactness and Arzela-Ascoli theorems then ensure that $ v $ is a solution to \eqref{bfamily} that belongs to $ C_{w}([-T,T]; H^1(\R)) $ with  $ v(0)=u_0 $ and that \eqref{cont2} holds. In particular, $ v_t\in L^\infty(]-T,T[; L^2(\R)) $ and thus $ v\in C([-T,T];L^2(\R)) $. Since $ v\in L^\infty(]-T,T[; H^{\frac{3}{2}-} (\R))$, this  actually implies that 
   $ v\in C([-T,T];  H^{\frac{3}{2}-}(\R)) $. Therefore, 
 $v$  belongs to the uniqueness class which ensures that $ v=u$.  
  \end{proof}
     \begin{remark} As noticed by Danchin in \cite{dan1} for the Camassa-Holm equation,  \eqref{by}  ensures  that smooth solution of \eqref{bfamily} emanating from  an initial data with an integrable initial momentum density satisfies 
  $$
  \frac{d}{dt} \int_{\R}  |y| =-\int_{\R} u \partial_x |y| - b \int_{\R}  u_x |y| 
  =(1-b)  \int_{\R} u_x |y|     \; .
  $$
   Therefore  by  \eqref{ut} and 
  H\"older inequality we get  
  \begin{align*}
  \frac{d}{dt} \|y(t)\|_{L^1} \le |b-1| \|y(t)\|^2_{L^1} 
  \end{align*}
  that leads to the a priori estimate 
  $$
  \|y(t)\|_{L^1} \le \frac{\|y_0\|_{L^1}}{1-|(b-1)t| \|y_0\|_{L^1}} ; .
  $$
  This shows that in the case $ b=1 $ we can actually replace the requirement $ u_0\in Y_+ $ by $ u_0 \in Y $ without changing the conclusion. Whereas in the case $ b\neq 1 $, replacing the requirement $ u_0\in Y_+ $ by $ u_0 \in Y $, exactly the same approach as the one to prove Proposition \ref{WP} leads to  the existence and uniqueness of a local solution   $u\in C^1([-T,T]; L^2(\R))\cap C([-T,T];W^{1,1} \cap H^1(\R)) $ 
  such that  $ y=(1-\partial_x^2)u \in C_{w}([-T,T]; \M)$ to \eqref{bfamily} with $T=T(\|y_0\|_{\M})>0 $.
  \end{remark} 

  \section{A rigidity result   for  $Y$-exponentially localized solution  of the $ b$-family moving to the right}\label{sect3}
 This section is devoted to the proof of Theorem \ref{rigidity}. 
 We will need the following lemma (see for instance \cite{Ifti})
 \begin{lemma}\label{BV}
 Let $\mu $  be a finite nonnegative measure on $  \R$ . Then $\mu $ is the sum of a nonnegative non atomic  measure $ \nu $ and a countable sum of positive Dirac measures (the discrete part of $\mu$). Moreover, for all $ \varepsilon>0 $ there exists $\delta>0 $  such that, if $I$ is an interval of length less than 
 $\delta$ , then $ \nu(I) \le \varepsilon $.
 \end{lemma}
 \subsection{Boundedness from above of  the momentum density support}
\begin{proposition} 
\label{pro3} Let $ b\in\R  $ and let $ u \in C(\R; Y_+) $ be a $Y$-localized global solution  with $ x_t \ge c_0>0$  of the $b$-family equation \eqref{bfamily} which is moreover uniformly exponentially decaying if $ b\neq 1$.
Assume furthermore that 
  $ \inf_{t\in \R} \|u(t)\|_{L^2} \ge \gamma_0 >0 $. There  exists $ r_0>0 $ such that\ for all $t\in \R$, it holds 
 \begin{equation}\label{pro3.1}
 \supp y(t,\cdot+x(t))\subset ]-\infty,r_0] ,
\end{equation}
and
\begin{equation}\label{pro3.2}
u(t,x(t)+r_0)=-u_x(t,x(t)+r_0)\ge \frac{e^{-2r_0}}{4 r_0} M(u) \, .
\end{equation}
 \end{proposition}
 \begin{proof}
 Clearly, it suffices to prove the result for $ t=0 $. 
  Let $ u \in C(\R; H^1) $,with $u-u_{xx}\in C_{w}(\R;\M_+) $,  be a $Y$-uniformly exponentially decaying solution to \eqref{bfamily} and   let $ \phi\in C^\infty(\R) $ with $ \phi\equiv 0 $ on $]-\infty,-1], \; \phi'\ge 0 $ and $ \phi\equiv 1  $ on $\R_+$. We claim that
    \begin{equation}\label{claim1}
  \dist{y(0)}{\phi(\cdot-(x(0)+r_0))}=0 
  \end{equation}
  which proves the result  since $ y\in \M_+ $.\\
  We prove \eqref{claim1} by contradiction. We approximate $u_0=u(0) $ by the sequence of smooth functions $ u_{0,n}=\rho_n\ast u_0 $ that belongs to $H^\infty(\R) \cap Y_+ $ so that \eqref{cont1}-\eqref{cont2} hold for any  $ T>0 $. We denote by $ u_n $ the solution to \eqref{bfamily} emanating from $ u_{0,n} $ and  by $ y_n=u_n-u_{n,xx}$ its momentum density. Let us recall that Proposition \ref{smoothWP} ensures that 
  $ u_n\in C(\R;H^\infty(\R)) $ and $ y_n\in C_{w}((\R;L^1(\R)) $.
We fix $ T>0 $ and we take $ n_0\in \N $ large enough so that for all $ n\ge n_0 $, 
\begin{equation}\label{approxu}
\|u_n -u \|_{L^\infty(]-T,T[; H^1)} <  \frac{1}{10}\min(c_0, M(u))
\end{equation}
and 
\begin{equation}\label{approxy0}
\|y_{0,n} -y_0 \|_{\mathcal{M}} < \frac{\varepsilon_0}{2} \quad .
\end{equation}
As explain in the beginning of Section \ref{section2}, the $Y $-almost localization of $ u $ actually forces an almost localization in all $ L^p$ for $ p\in [1,+\infty] $. 
  Therefore there exists $ r_0>0 $ such that 
   \begin{equation}\label{hu1}
  \|u(t,\cdot +x(t)) \|_{L^1(]-r_0,r_0[^c)}+ \|u(t,\cdot +x(t)) \|_{L^\infty(]-r_0,r_0[^c)} +u(t,x(t)+x) \le  \frac{1}{10}\min(c_0,  M(u)), \quad \forall t\in\R \, .
  \end{equation}
  Combining this estimate with \eqref{approxu} we infer that for $ n\ge n_0 $, 
  \begin{equation}\label{hu1n}
  u_n(t,x(t)+x) \le  \frac{1}{5}\min(c_0, M(u)), \forall (|x|,t)\in [r_0,+\infty[\times [-T,T] \, .
  \end{equation}
 
 Now, we introduce the flow $ q_n$ associated with $ u_n $ defined by 
 \begin{equation}\label{defqn}
  \left\{ 
  \begin{array}{rcll}
  q_{n,t}(t,x) & = & u_n(t,q_n(t,x))\, &, \; (t,x)\in \R^2\\
  q_n(0,x) & =& x\, & , \; x\in\R \; 
  \end{array}
  \right. .
  \end{equation} 
  Following \cite{C}, we know that for any $ t\in \R $,
  \begin{equation}\label{yy}
  y_n(0,x)=y_n(t,q_n(t,x)) q_{n,x}(t,x)^b
  \end{equation}
    Indeed, on one hand,  \eqref{by} clearly ensures that 
 $$
 \frac{\partial}{\partial t} \Bigl( y_n(t,q_n(t,x))e^{b\int_0^t u_{n,x}(s, q_n(s,x))\, ds} \Bigr) =0 
$$ 
and, on the other hand, \eqref{defqn}  ensures that $ q_{n,x}(0,x)=1 $, $\forall x\in \R $, and 
\begin{equation}\label{formula1}
\frac{\partial}{\partial t} q_{n,x}(t,x)=u_{n,x}(t,q_n(t,x)) q_{n,x}(t,x) \Rightarrow q_{n,x}(t,x) =\exp\Bigl( \int_0^t u_{n,x}(s,q_n(s,x))\,ds\Bigr) \; .
\end{equation}
  We claim that for all $ n\ge n_0 $ and $ t\in [-T,0] $ , 
   \begin{equation}\label{claim2}
 q_n(t,x(0)+r_0) -x(t) \ge r_0+\frac{c_0}{2} |t|  \; .
  \end{equation}
  Indeed, fixing $ n\ge n_0 $, in view of  \eqref{hu1n} and the continuity of $ u_n $ there exists $ t_0\in [-T,0[$ such that for all $t\in [t_0,0] $, 
  $$
  u_n(t,q_n(t,x(0)+r_0))\le \frac{c_0}{4}
  $$
  and thus according to \eqref{defqn}, for all $ t\in [t_0,0] $, 
  $$
  \frac{d}{dt} q_n(t,x(0)+r_0) \le \frac{c_0}{4} 
  $$
  which leads to 
  $$
  q_n(t,x(0)+r_0)-x(t) \ge r_0+\frac{c_0}{2} |t|, \quad  t\in [t_0,0]\; .
  $$
  This proves \eqref{claim2} by a continuity argument.
  
   Now, in the case $ b\neq 1$, we thus deduce from the uniform exponential decay of $ u $ that for all $ t\in [-T,0]  $ and all $ x\ge 0 $,
  \begin{equation}
  u(t,q_n(t,x(0)+r_0+x)  \le C \exp \Bigl( -\beta (r_0+ c_0 |t|/2 )\Bigr)\\
  \end{equation}
  Therefore, in view of \eqref{cont1} and \eqref{dodo}, there exists $n_1\ge n_0 $ such that for all $ t\in [-T,0] $ and all $ x\ge 0 $,
  \begin{equation}
  u_n(t,q_n(t,x(0)+r_0+x) +|u_{n,x}(t,q_n(t,x(0)+r_0+x)| \le 4 C \exp \Bigl( -\beta(r_0+ c_0 |t|/2 )\Bigr)\\
  \end{equation}
  The formula (see \eqref{formula1})
  \begin{equation}\label{formula}
  q_{n,x}(t,x)=\exp \Bigl( -\int_t^0 u_{n,x}(s,q_n(s,x))\, ds \Bigr) 
  \end{equation}
   thus ensures that $ \forall t\in [-T,0] , \; \forall x\ge 0 $ and $ \forall n\ge n_0 $, 
  $$
   \exp\Bigl( -4C \int_{-T}^0 e^{ -\beta (r_0+ c_0 |s|/2)}\, ds \Bigr) \le  q_{n,x}(t,x(0)+r_0+x) \le \exp\Bigl( 4C \int_{-T}^0 e^{ -\beta (r_0+ c_0 |s|/2)}\, ds \Bigr)
  $$
  Setting $ C_0:= e^{\frac{8  C e^{-\beta r_0}}{\beta c_0}} $ this leads, in the case $ b\neq 1$,  to 
  \begin{equation}\label{to}
  \frac{1}{C_0} \le  q_{n,x}(t,x(0)+r_0+x) \le C_0 \, , \; \forall t\in [-T,0]\; .
  \end{equation}
  Now, we claim that for any $ b\in\R $ and any $ n\ge n_1 $  it holds 
  \begin{equation}\label{toto}
\int_{x(0)+r_0}^{+\infty} y_n(0,x)\, dx \le C_0^{b-1}\int_{x(t)+r_0+c_0 |t|/2}^{+\infty} y_n(t,z)\, dz\, , \; \forall t\le [-T,0]  \; .
\end{equation}
Letting $ n\to +\infty $ using \eqref{cont2} and then letting $ T \to \infty $,  this ensures that 
$$
\dist{y(0)}{ \phi(\cdot-x(t)-r_0)} \le C_0^{b-1} \dist{y(t)}{ \phi(\cdot-x(t)-r_0-c_0 |t|/2+1)}\, , \; \forall t\le 0  \; 
$$
which proves \eqref{claim1} since   the $ Y $-uniform localization of $ u $ forces the right-hand side member to goes to $0 $ as $ t\to -\infty $.
Therefore, to complete the proof of  \eqref{pro3.1}, it remains to prove \eqref{toto}.
First, it follows from \eqref{yy} that for any $ t\le 0 $ and  any $ r_0'>r_0 $, 
$$
\int_{x(0)+r_0}^{x(0)+r_0'} y_n(0,x)\, dx=\int_{x(0)+r_0}^{x(0)+r_0'} y_n(t,q_n(t,x))q_{n,x}(t,x)^b \, dx
$$
and \eqref{to} leads  for any $ b\in \R $ to 
$$
\int_{x(0)+r_0}^{x(0)+r_0'} y_n(0,x)\, dx\le C_0^{b-1} \int_{x(0)+r_0}^{x(0)+r_0'} y_n(t,q_n(t,x)) q_{n,x}(t,x)\, dx 
$$ 
The change of variables $ z=q_n(t,x) $ then  yields 
$$
\int_{x(0)+r_0}^{x(0)+r_0'} y_n(0,x)\, dx \le C_0^{b-1}  \int_{q_n(t,x(0)+r_0)}^{q_n(t,x(0)+r_0')} y_n(t,z)\, dz $$
and \eqref{toto} then follows from \eqref{claim2}  by letting $ r_0' $ tend to $ +\infty $.

Let us now prove \eqref{pro3.2}.  We first notice that  thanks to \eqref{hu1} and  the conservation of
 $ M(u)=\langle y, 1\rangle = \int_{\R} u $,  
 it holds 
\begin{equation}\label{max}
\max_{[-r_0,r_0]} u(t,\cdot+x(t)) \ge \frac{1}{2r_0} \|u(t,\cdot+x(t))\|_{L^1(]-r_0,r_0[)} \ge \frac{M(u)}{4 r_0}\; .
\end{equation}
  But, since $ u_x\ge -u $ on $\R^2$, for any $ (t,x_0)\in\R^2 $ it holds
 $$
 u(t,x) \le u(t,x_0) e^{-x+x_0}, \quad \forall x\le x_0 \; .
 $$
Applying this estimate with  $ x_0=x(t)+r_0 $ we obtain that 
 $$
 u(t,x(t)+r_0) \ge  \max_{[-r_0,r_0]} u(t, \cdot+x(t))  e^{-2r_0}
 $$
which, combined with \eqref{max} yields \eqref{pro3.2}.
 \end{proof}
 \begin{remark}
 It is worth pointing out that $ b=1 $ is a particular case here. Indeed, in this case we do not need the $Y$-uniform exponential localization of $ u $
  but only a $ Y $-uniform almost localization since we do not need \eqref{to} anymore.
 \end{remark}
 
\subsection{Study of the supremum of the support of $ y$}
We define 
$$
x_+(t)=\inf \{ x\in\R, \, \supp y(t)\subset  ]-\infty, x(t)+x] \} 
$$
In the sequel we set 
$$
\alpha_0:=\frac{e^{-2r_0}}{4 r_0}M(u)
$$
to simplify the expressions. According to Proposition \ref{pro3}, $  t\mapsto x_+(t) $ is well defined with values in $ ]-\infty, r_0] $ and 
\begin{equation}
u(t,x(t)+x_+(t))=-u_x(t,x(t)+x_+(t))\ge \alpha_0 \, .
\end{equation}
The following lemma proved in \cite{L}  ensures that $t\mapsto x(t)+x_+(t)$ is an integral line of $ u $. 
\begin{lemma} \label{lemmaq}
For all $ t\in \R$, it holds
\begin{equation} \label{qq}
x(t)+x_+(t)=q(t,x(0)+x_+(0)) \; .
\end{equation}
where $ q(\cdot,\cdot)$ is defined by 
 \begin{equation}\label{defq}
  \left\{ 
  \begin{array}{rcll}
  q_t (t,x) & = & u(t,q(t,x))\, &, \; (t,x)\in \R^2\\
  q(0,x) & =& x\, & , \; x\in\R \; 
  \end{array}
  \right. .
  \end{equation} 
\end{lemma}
 In the sequel we define $ q^* \, :\, \R \to \R $ by 
 \begin{equation}\label{defq*}
 q^*(t)=q(t,x(0)+x_+(0))=x(t)+x_+(t), \quad \forall t\in \R \; .
 \end{equation}

\begin{proposition}\label{propa}
Assume that $ u$ satisfies the hypotheses of Proposition \ref{pro3} with $b\ge 1 $. Let $ a \, :\, \R\to \R $ be the function defined by 
\begin{equation}\label{defa}
a(t)=u_x(t,q^*(t)-)-u_x(t,q^*(t)+), \quad \forall t\in\R \, .
\end{equation}
Then $ a(\cdot) $ is  a bounded non decreasing  derivable function on $ \R $ with values in $[\frac{\alpha_0}{8}, M(u)]$  such that 
\begin{equation}\label{esta}
a'(t)=\frac{1}{2}(u^2-u_x^2)(t,q^*(t)-) , \; \forall t\in \R .
\end{equation}
\end{proposition}
\begin{proof}
First, the fact that $ a(t) \le M(u)$ follows  from the conservation of $ M $ together with Young convolution inequalities since
 $ u=\frac{1}{2} e^{-|x|} \ast y $. 
To prove that $ a(t) \ge \frac{\alpha_0}{8}$, 
 we proceed by contradiction. So let us assume that there exists $ t_0\in\R $ such that  $ a(t_0)<\alpha_0/8 $. Since 
  $ y(t_0)\in \M_+ $ with $ \supp y(t_0)\subset ]-\infty, q^*(t_0)]$, according to Lemma \ref{BV} we must have
$$
\lim_{z\nearrow q^*(t_0)} \|y(t_0)\|_{{\mathcal M}(]z,+\infty[)}<\frac{\alpha_0}{8} \;.
$$
Without loss of generality we can assume that $ t_0=0 $ and thus  there exists $ \beta_0>0 $ such that 
\begin{equation}\label{toto1}
 \|y(0)\|_{{\mathcal M}(]q^*(0))-\beta_0,+\infty[)}<\frac{\alpha_0}{8} \; .
\end{equation}
By convoluting $ u_0$ by $ \rho_n $ (see \eqref{rho}), for some $ n\ge 0 $, we can  approach $ u_0 $ by a smooth function $\tilde{u}_0 \in Y_+\cap H^\infty(\R) $.  
 Taking $ n $ large enough, we may  assume that there exists  $ \tilde{x}_+>x_+(0)$ close to $ x_+(0) $,  such that 
  \begin{equation}\label{hu21}
 \tilde{y}_0=(1-\partial_x^2) \tilde{u}_0 \equiv 0 \mbox{ on }[x(0)+\tilde{x}_{+},+\infty[
 \end{equation}
 and
  \begin{equation}\label{hu2}
  \|\tilde{y}_0\|_{L^1(]x(0)+\tilde{x}_+-\beta_0,+\infty[)}\le  \frac{\alpha_0}{8}+ \frac{\alpha_0}{2^6}\;, 
 \end{equation}
 where $\tilde{y}_0=\tilde{u}_0-\tilde{u}_{0,xx} $. 
Moreover, defining $ \tilde{q}_2 \, :\, \R\to \R $ by 
$$
\tilde{q}_2(t)=\tilde{q}(t,x(0)+\tilde{x}_+)
$$
where $ \tilde{q}(\cdot,\cdot) $ is defined by   \eqref{defq} with $ u $ replaced by $ \tilde{u}$, 
   \eqref{cont1} enables us  to assume that   the emanating solution $ \tilde{u } $ satisfies 
 \begin{equation}\label{appo}
 \|u(t) - \tilde{u}(t)\|_{H^1} \le \frac{\alpha_0}{2^6}
 \end{equation}
 and 
 \begin{equation}\label{appo2}
 |q^*(t)-\tilde{q}_2(t)|<\frac{\alpha_0}{2^6 M(u)}
 \end{equation}
 for all $ t\in [-t_1,t_1] $ with $ t_1 >0 $ to specified later. 
  It is worth noticing that \eqref{appo}-\eqref{appo2}, \eqref{hu21}, \eqref{yy}, \eqref{pro3.2} and the mean-value theorem  - recall that by Young inequalities $ |u_x|\le 
  \frac{1}{2} M(u) $ - ensure that 
 \begin{equation}\label{appo3}
-\tilde{u}_x(t, \tilde{q}_2(t)) = \tilde{u}(t, \tilde{q}_2(t)) \ge (1-2^{-5}) \alpha_0\; \quad \forall t\in [-t_1,t_1] \; .
 \end{equation}

We claim that  for all $ t\in [-t_1,0] $ it holds 
\begin{equation}\label{claim3}
\tilde{u}_x(t,x)\le -\frac{3\alpha_0}{4}\quad \mbox{on}\quad  [\tilde{q}_1(t), \tilde{q}_2(t)] \, , 
\end{equation}
where $ \tilde{q}_1(t) $ is defined by  $ \tilde{q}_1(t)=\tilde{q}(t,x(0)+\tilde{x}_+-\beta_0)$.\\
To see this, for $ \gamma>0 $, we set 
$$
A_{\gamma}=\{t\in \R_-\,  /\, \forall \tau\in [t,0], \; u_x(\tau,x) 	< -\gamma \; \mbox{on} \; [\tilde{q}_1(\tau), \tilde{q}_2(\tau)] \, \} \; .
$$
Recalling \eqref{pro3.2}, \eqref{hu2}, \eqref{appo3} and  that $ \tilde{u} \ge 0 $ , we get for $ 0\le \beta\le \beta_0$, 
\begin{eqnarray*}
\tilde{u}_x(0,x(0)+\tilde{x}_{+}-\beta) &  \le &  \tilde{u}_x(0,x(0)+\tilde{x}_{+}) +  \|\tilde{y}_0\|_{L^1(]x(0)+\tilde{x}_{+}-\beta_0,+\infty[)} \\
 &\le & -\alpha_0+\frac{\alpha_0}{2^5}+\frac{\alpha_0}{8} +\frac{\alpha_0}{2^5}< -\frac{3\alpha_0}{4}\; ,
\end{eqnarray*}
which ensures that $  A_\frac{3\alpha_0}{4} $ is non empty. By  a continuity argument, it thus suffices to prove that $  A_\frac{\alpha_0}{2} \subset    A_\frac{3\alpha_0}{4} $.
First we notice that for any $ t\in  A_\frac{\alpha_0}{2}  $ and any $ x\in [\tilde{q}_1(t), \tilde{q}_2(t)]$, the definition of  $ A_\frac{\alpha_0}{2}  $ ensures that 
$$
\tilde{q}_x(t,x)=\exp\Bigl( -\int_t^0 \tilde{u}_x(\tau,\tilde{q} (\tau,x))\, d\tau \Bigr)\ge 1 \; ,
$$
where $ \tilde{q}(\cdot,\cdot) $ is the flow associated to $ \tilde{u}$  by \eqref{defq}.
Therefore, $ \tilde{u}\ge 0 $, $ \tilde{y}\ge 0 $, a  change of variables, \eqref{yy}  and \eqref{hu2} ensure that for any $ x\in [\tilde{q}_1(t), \tilde{q}_2(t)]$,
$$
\int_{x}^{\tilde{q}_2(t)} \tilde{u}_{xx} (t,s) \, ds\ge - \int_{x}^{\tilde{q}_2(t)}\tilde{y} (t,s) \, ds\ge - \int_{\tilde{q}_1(t)}^{\tilde{q}_2(t)} \tilde{y} (t,s) \, ds=-\int_{x(0)+\tilde{x}_+ -\beta_0}^{x(0)+\tilde{x}_+} \tilde{y}(t,\tilde{q}(t,s))\tilde{q}_x(t,s) \, ds \;.
$$
Using that $ b\ge 1 $ this leads to
$$
\int_{x}^{\tilde{q}_2(t)} \tilde{u}_{xx} (t,s) \, ds\ge  -\int_{x(0)+\tilde{x}_+ -\beta_0}^{x(0)+\tilde{x}_+} \tilde{y}(t,\tilde{q}(t,s))\tilde{q}_x(t,s)^b \, ds=   -\int_{x(0)+\tilde{x}_+ -\beta_0}^{x(0)+\tilde{x}_+} \tilde{y}_0(s) \, ds\ge -\frac{\alpha_0}{8}-\frac{\alpha_0}{2^6} 
$$
and \eqref{appo3}  yields
$$
\tilde{u}_x(t,x) = \tilde{u}_x(t,q_2(t))-\int_{x}^{q_2(t)} \tilde{u}_{xx} (t,s) \, ds \le -\alpha_0+\frac{\alpha_0}{8}+\frac{\alpha_0}{2^4}<-\frac{3\alpha_0}{4} \; ,
$$
which proves the desired result.

 We deduce from \eqref{claim3} that $ \forall t\in [-t_1, 0] $,
\begin{eqnarray*}
\frac{d}{dt} (\tilde{q}_2(t)-\tilde{q}_1(t)) & = & \tilde{u}(\tilde{q}_2(t))-\tilde{u}(\tilde{q}_1(t)) \\
& =& \int_{\tilde{q}_1(t)}^{\tilde{q}_2(t)} \tilde{u}_x(t,s) \, ds \\
& \le & -\frac{\alpha_0}{2} (\tilde{q}_2(t)-\tilde{q}_1(t)) \; .
\end{eqnarray*}
Therefore,
$$
(\tilde{q}_2-\tilde{q}_1)(t)\ge (\tilde{q}_2-\tilde{q}_1)(0) e^{-\frac{\alpha_0}{2} t} = \beta e^{-\frac{\alpha_0}{2} t}  \; .
$$
 On the other hand,  since  according to \eqref{appo3} and \eqref{claim3}, 
$ \tilde{u}(t, \tilde{q}_2(t))\ge 2 \alpha_0/3 $ and $\tilde{u}_x \le 0 $ on $]\tilde{q}_1(t), \tilde{q}_2(t)[ $, we deduce that   
$$
\tilde{u}(t,\tilde{q}_1(t))\ge \tilde{u}(t,\tilde{q}_2(t))\ge 2\alpha_0/3 , \quad \quad \mbox{ on } [-t_1,0] \; .
$$
Coming back to the solution $ u $ emanating from $ u_0 $, it follows from  \eqref{appo}  that 
$$
\min\Bigr(u(t,\tilde{q}_1(t_1)),u(t,\tilde{q}_2(t_1))\Bigr)\ge \frac{\alpha_0}{2} \mbox{ with } (\tilde{q}_2-\tilde{q}_1)(t_1)\ge \beta e^{-\frac{\alpha_0}{2} t} \; , \quad
\forall t\in [-t_1,0]\, .
$$
Taking $ t_1>0 $ large enough, this  contradicts the $Y$-almost localization of $ u$ which proves that $ a(t)\ge \frac{\alpha_0}{8}$ and thus $ u_x(t,\cdot) $ has got a jump at 
 $ x(t)+x_+(t) $.
 
It is worth noticing that,according Lemma \ref{BV}, this ensures that for all $ t\in\R $, one can decompose $ y(t)$ as 
\begin{equation}\label{decompositiony}
y(t)= \nu(t)+ a(t) \delta_{x(t)+x_+(t)}+\sum_{i=1}^\infty a_i(t) \delta_{x_i(t)}
\end{equation}
where $ \nu(t) $ is a non negative non atomic measure with $ \nu(t)\equiv 0 $ on $  ]x(t)+x_+(t), +\infty[ $,  $\{a_i\}_{n\ge 1 } \subset (\R_+)^{\N} $  with $ \sum_{i=1}^\infty a_i(t)<\infty $ and $ x_i(t)<x(t)+x_+(t) $ for all $ i\in \N^* $.
It remains to prove that for all couple of real numbers $(t_1,t_2) $ with  $ t_1<t_2 $,
\begin{equation}\label{tt}
a(t_2)-a(t_1)=\frac{1}{2}\int_{t_1}^{t_2} (u^2-u_x^2)(\tau,q^*(\tau)-) \, d\tau
\end{equation}
Indeed, since $ |u_x|\le u $ and  $u\in L^\infty(\R^2) $ this will force $ a$ to be non decreasing and derivable on $ \R $.
 Let $ \phi \, :\, \R\to \R_+$ be a non decreasing  $C^\infty $-function such that $ \supp \phi\subset [-1,+\infty[ $ and $ \phi\equiv 1 $ on $\R_+ $. We set $ \phi_\varepsilon=\phi(\frac{\cdot}{\varepsilon}) $. Since $ u $ is continuous and $y(t, \cdot) =0 $ on $ ]x(t)+x_+(t),+\infty[ $ it follows from \eqref{decompositiony} that for all $ t\in \R $,
 $$
 a(t)=\lim_{\varepsilon\searrow 0} \langle y(t),\phi_\varepsilon (\cdot -q^*(t))\rangle \; .
 $$
 Without loss of generality, it suffices to prove \eqref{tt} for $ t_1=0 $ and $ t_2=t\in ]0,1[$.  Let $ \beta>0 $  be fixed, we claim that there exists $ \varepsilon_0>0 $ such that for all $ 0<\varepsilon<\varepsilon_0 $, 
  \begin{equation} \label{claimb}
\Bigl|   \langle y(t),\phi_\varepsilon (\cdot -q_*(t))\rangle-\langle y(0),\phi_\varepsilon (\cdot -q^*(0))\rangle-\frac{1}{2}\int_0^t \int_{\R}  (u^2-u_x^2)(\tau,q^*(t)+\varepsilon z ) 
\phi'(z)\, dz \, d\tau \Bigr|
\le  \beta , \quad \forall t\in ]0,1[
  \end{equation}
  Passing to the limit as $ \varepsilon $ tends to $ 0 $, this  leads to the desired result. Indeed, since $(u^2-u_x^2)(\tau,\cdot) \in BV(\R) $ and 
   $ \phi'\equiv 0 $ on $\R_+$, for any fixed $ (\tau,z) $,  it is clear that 
  $$
  (u^2-u_x^2)(\tau,q^*(\tau)+\varepsilon z ) \phi'(z) \tendsto{\varepsilon\to 0} (u^2-u_x^2)(\tau, q^*(\tau)-) \phi'(z)
  $$
  and,  since it is dominated by $ 2 M(u)^2 \phi' $, the dominated convergence theorem leads to 
  \begin{eqnarray*}
  \int_0^t \int_{\R}  (u^2-u_x^2)(\tau,q^*(t)+\varepsilon z ) 
\phi'(z))\, dz \, d\tau & \tendsto{\varepsilon \to 0}  & \int_0^t \int_{\R}  (u^2-u_x^2)(\tau,q^*(\tau)- ) 
\phi'(z))\, dz \, d\tau \\
&  &= \int_{0}^{t} (u^2-u_x^2)(\tau,q^*(\tau)-) \, d\tau \; .
  \end{eqnarray*}
 To prove \eqref{claimb} we first notice that  according to \eqref{decompositiony} for any $ \alpha>0 $ there exists $ \gamma(\alpha) >0 $ such that 
  \begin{equation}\label{uu}
 \|y\|_{{\mathcal M}(]q^*(0)-\gamma(\alpha),q^*(0)[)}<\alpha \; .
 \end{equation}
   We approximate again $ u(0) $ by a sequence $ \{u_{0,n}\} \subset H^{\infty}(\R)\cap Y_+ $ such that $ M(u_{0,n}) \le  2 M(u)$ and we ask that 
 \begin{equation}\label{approx}
 \|y(0)-y_{0,n}\|_{{\mathcal M}(\R)} \le e^{(1-b) M(u)}\beta/4\; .
 \end{equation}
 where $ y_{0,n}=u_{0,n}-\partial_x^2 u_{0,n} $. We again denote respectively  by $ u_n $ and $ y_n $,  the solution to \eqref{bfamily} emanating from $ u_{0,n} $ and its momentum density $ u_n-u_{n,xx}$.  Let now $ q_n^* \; :\R\to \R $ the integral line of $ u_n $ defined by 
  $q_n^*(t)=q_n(t,q^*(0)) $ where $ q_n $ is defined in \eqref{defqn}.
 On account of \eqref{by}, it holds
\begin{eqnarray}
\frac{d}{dt} \int_{\R} y_n \phi_\varepsilon (\cdot -q_n^*(t)) &= &-u_n(t,q_n^*(t)) \int_{\R} y_n \phi_\varepsilon'  -\int_{\R} \partial_x (y_n u_n) \phi_{\varepsilon} -(b-1) \int_{\R} y _n u_{n,x} \phi_{\varepsilon}\nonumber \\
 & = & \int_{\R} \Bigl[ u_n(t,\cdot)-u_n(t,q^*(t))\Bigr] y_n(t,\cdot) \phi_{\varepsilon}'+\frac{b-1}{2} \int_{\R} (u_n^2(t,\cdot)-u_{n,x}^2(t,\cdot)) \phi_{\varepsilon}' \nonumber\\
  & = &\frac{1}{\varepsilon} \int_{\R} \Bigl[ u_n(t,\cdot)-u_n(t,q^*(t))\Bigr] y_n(t,\cdot) \phi'\Bigl(\frac{\cdot -q_n^*(t)}{\varepsilon}\Bigr)\nonumber \\
   & &+\frac{b-1}{2} \int_{\R} (u_n^2-u_{n,x}^2)(t,q_n^*(t)+\varepsilon z) \phi'(z) \, dz \nonumber\\
 & = & I_t^{\varepsilon,n}+ II_t^{\varepsilon,n}\; .\label{fd}
\end{eqnarray} 
Since, according to \eqref{ut}, $ |u_{n,x}|\le \frac{1}{2} M(u)$, 
\begin{eqnarray*}
|I_t^{\varepsilon,n}| & \le &  \frac{M(u)}{\varepsilon} \int_{\R} |x-q_n^*(t) | y_n(t,x) \phi'(\frac{x-q_n^*(t)}{\varepsilon}) \, dx  \\
& \le & M(u)\int_{\R}  y_n(t,x) \phi'(\frac{x-q_n^*(t)}{\varepsilon}) \, dx
\end{eqnarray*}
Now,  in view of \eqref{formula} we easily get 
\begin{equation}\label{gg}
e^{-M(u)}\le q_{n,x}(t,z) \le e^{M(u)} , \quad \forall (t,z) \in ]-1,1[\times \R \, 
\end{equation}
and  the change of variables $ x=q_n(t,z) $ together with  the identity \eqref{yy}  lead to 
\begin{eqnarray*}
 \int_{\R}  y_n(t,x) \phi'(\frac{x-q_n^*(t)}{\varepsilon}) \, dx & = &  \int_{\R} y_n(t,q_n(t,z))q_{n,x}(t,z) \phi_\varepsilon' (q_n(t,z)-q_n^*(t)) \, dx \\
 & \le & e^{(b-1)M(u)} \int_{\R} y_n(t,q_n(t,z)) (q_{n,x}(t,z))^b \phi_\varepsilon' (q_n(t,z)-q_n^*(t)) \, dz \\
 & \le & e^{(b-1)M(u)} \int_{\R}y_n(0,z) \phi_\varepsilon' (q_n(t,z)-q_n^*(t)) \, dz\; .
\end{eqnarray*}
 But, making use of the mean value theorem,  \eqref{gg} and the definition of $ \phi_{\varepsilon}$,  we obtain that, for any $ t \in [0,1]$,  $ z\mapsto  \phi_\varepsilon' (q_n(t,z)-q_n^*(t)) $ is  supported in an interval of length at most $ \varepsilon e^{M(u)} $. Therefore,  according to \eqref{uu} and  \eqref{approx}, setting $  \varepsilon_0=e^{-M(u)} \gamma(\frac{\beta}{2} e^{(1-b)\M(u)})$, it follows that for all $ 0<\varepsilon<\varepsilon_0 $ and all $ n\in \N $, 
 \begin{equation}\label{ese1}
 \int_0^t |I_\tau^{\varepsilon,n} |d\tau \le 3\beta/4 \; .
 \end{equation}
 To estimate the contribution of $ II_t^{\varepsilon,n}$ we first notice that thanks \eqref{cont1} it holds 
 $$
 u_{n,x} \to u_x \mbox{ in } C([-1,1]; L^2(\R))
 $$
 and for all $ t\in [-1,1]$,  Helly's theorem ensures that
 $$
u_{n,x}(t,\cdot) \to u_x(t,\cdot) \; \mbox{ a.e. on } \R \; .
$$
 Hence, for any fixed $ t\in [-1,1] $ there exists a set $ \Omega_t \subset \R $ of  Lebesgue measure zero such that $ u_x(t) $ is continuous at every point $x\in \R/\Omega_t $ and 
 $$
 u_{n,x}(t,x)\to u_x(t,x) \;, \quad \forall x\in \R/\Omega_t \; .
 $$
 Since $q_n^*(t) \to q^*(t) $, it follows that 
  $$
 u_{n,x}(t,q_n^*(t)+x)\to u_x(t,q^*(t)+x) \;, \quad \forall x\in \R/\tau_{q^*(t)}(\Omega_t) \; .
 $$
 where for any set $ \Lambda\subset \R $ and any $ a\in \R $, $ \tau_a(\Lambda)=\{x-a, a\in \Lambda\}$.\\
 Since the integrand in $ II_t^{\varepsilon,n}$ is bounded by $ M(u)\phi'\in L^1(\R) $, it follows from  Lebesgue dominated convergence theorem that for any $ t\in [-1,1] $, 
 $$
  II_t^{\varepsilon,n} \tendsto{n\to \infty} \frac{1}{2}\int_{\R} (u^2-u_x^2)(t, q^*(t)+\varepsilon z) \phi'(z) \, dz \;.
  $$
  Therefore, invoking again Lebesgue dominated convergence theorem, but on $ ]0,t[ $, keeping in mind that $\{|u_n|\} | $ and $ \{|u_{n,x}|\} $ are uniformly bounded on $ \R^2 $ by $ M(u)$, we finally deduce that for any fixed $ t\in ]0,1[ $, 
  \begin{equation}\label{ese2}
  \int_0^t II_{\tau}^{\varepsilon,n} \, d\tau \tendsto{n\to \infty} \frac{1}{2}\int_0^t \int_{\R}  (u^2-u_x^2)(\tau,q^*(t)+\varepsilon z ) 
\phi'(z))\, dz \, d\tau 
  \end{equation}
 Now, we fix $ t \in ]0,1[ $ and $ \varepsilon \in ]0,\varepsilon_0[ $. According to the convergence result \eqref{cont2}, for $ n $ large enough it holds 
  $$
  | \langle y_n(t)-y(t),\phi_\varepsilon (\cdot -q_*(t))\rangle|+|\langle y_n(0-y(0),\phi_\varepsilon (\cdot -q_*(0))\rangle| \le \beta/4 \; .
  $$
 which together with \eqref{fd} and \eqref{ese1}-\eqref{ese2} prove the claim \eqref{claimb}.
\end{proof}
\begin{lemma} \label{lem43}
There exists $(a_-,a_+)\in [\frac{\alpha_0}{8}, M(u)]^2 $, with $ a_-\le a_+ $ such that 
\begin{eqnarray}
\lim_{t\to  +\infty} u(t,x(t)+x_+(t)) & =& \lim_{t\to +\infty} a(t)/2=a_{+}/2\; ,\\
\lim_{t\to  -\infty} u(t,x(t)+x_+(t)) & = & \lim_{t\to -\infty} a(t)/2=a_{-}/2\; ,
\end{eqnarray}
 \end{lemma}
\begin{proof} The existence of the limits at $ \mp \infty $ for $ a(\cdot) $ follows from the monotonicity of $ a(\cdot) $. 
Now, in view of Proposition \ref{propa}, for all $ t\in\R $,  
\begin{eqnarray}
0 \le a'(t)=\frac{1}{2}(u^2-u_x^2)(t,x(t)+x_+(t)-) & =  & \frac{a(t)}{2}(u-u_x)(t,x(t)+x_+(t)-)\nonumber \\
 &= & \frac{a(t)}{2} (2u(t,x(t)+x_+(t))-a(t)) \; .\label{a'}
\end{eqnarray}
Therefore, since $ a $ takes values in $[\alpha_0/8, M(u)] $,  it remains to prove that $ a'(t)\to 0 $ as $ t\to \pm\infty $. 
Since
$$
\int_{\R} a'(\tau) \, d\tau 	<\infty \; ,
$$
 the desired result will follow if  $ a' $ is Lipschitz on $ \R $. But this is not too hard to check. Indeed, first from \eqref{esta} we have for all $t\in\R $, $ |a(t)-a(0)|\le 2t \|u_0\|_{H^1} $ and thus $ t\mapsto a(t) $ is clearly Lipschitz on $ \R $. Second, since $ x(t)+x_+(t)=q^*(t) $, it holds 
$$
\frac{d}{dt} u(t,x(t)+x_+(t))= u(t,q^*(t)) u_x(t,q^*(t))+u_t(t,q^*(t)) \; .
$$
But, $ \sup_{(t,x)\in \R^2} |u u_x |\le M(u)^2 $
 and 
 \begin{eqnarray*}
 \sup_{(t,x)\in\R^2}  |u_t| & \le &  \sup_{(t,x)\in\R^2} |u u_x |+  \sup_{(t,x)\in\R^2} \Bigl| (1-\partial_x^2)^{-1}\partial_x (
 \frac{b}{2}u^2+\frac{3-b}{2} u_x^2) \Bigr| \\
 & \lesssim &M(u)^2   + \sup_{t\in\R} \| \frac{b}{2}u^2(t)+\frac{3-b}{2} u_x^2(t)\|_{L^1_x} \\
& \lesssim & M(u)^2 \; .
 \end{eqnarray*}
 Therefore $ t\mapsto u(t,x(t)+x_+(t)) $ is also Lipschitz on $ \R $ which achieves the proof thanks to \eqref{a'}.
\end{proof}
\noindent
\subsection{ End of the proof ot Theorem \ref{rigidity}.} 
In this subsection, we conclude by proving that the jump of $ u_x(0,\cdot) $ at $ x(0)+x_+(0) $ is equal to $ -2u(0,x(0)+x_+(0))$. This saturates  for all $ v\in Y_+$, the relation between the jump of $ v_x $  and the value of $ v$ at a point $ \xi \in\R $ and forces $ u(0,\cdot) $ to be equal to $ u(0,x(0)+x_+(0)) \varphi(\cdot- x(0)+x_+(0))$.

We use the invariance of the (CH) equation under the transformation $ (t,x) \mapsto (-t,-x) $. This invariance ensures that  $ v(t,x)=u(-t,-x) $ is also a solution of the (C-H) equation that belongs to $ C(\R; H^1(\R) $, with $ u-u_{xx}\in C_{w}(\R;\M_+) $  and  shares the property of $Y$-almost localization with $ x(\cdot) $ replaced by $ -x(-\cdot) $ and the same fonction $ \varepsilon \mapsto R_\varepsilon $  (See Definition \ref{defYlocalized}). 
Therefore, by applying Propositions \ref{pro3}, \ref{propa} and Lemma \ref{lemmaq} for $ v$ we infer that there exists a $C^1 $-function $ x_- \, :\, \R \mapsto ]-\infty,r_0] $ and a derivable non decreasing function $\tilde{a} \, :\, \R \to [\alpha_0/8, M(u) ]$ with $ \lim_{t\to\mp \infty}
\tilde{a}(t)=\tilde{a}_{\mp} $ 
such that 
\begin{equation}\label{defatilde}
\tilde{a}(t)=v_x(t,(-x(-t)+x_+(t) )+)-v_x(t,(-x(-t)+x_+(t) )-), \quad \forall t\in\R \, .
\end{equation}
Moreover,
$$
\lim_{t\to  \mp\infty} v(t,-x(-t)+x_+(t))  = \lim_{t\to \mp\infty} \tilde{a}(t)/2=\tilde{a}_{\mp}/2\; .\\
$$
Coming back to $ u $ this ensures that 
\begin{eqnarray}
\lim_{t\to  +\infty} u(t,x(t)-x_-(-t)) & =& \lim_{t\to -\infty} \tilde{a}(t)/2=\tilde{a}_{-}/2\; ,\\
\lim_{t\to  -\infty} u(t,x(t)-x_-(-t)) & = & \lim_{t\to +\infty} \tilde{a}(t)/2=\tilde{a}_{+}/2\; ,
\end{eqnarray}
At this stage let us underline that  since
$$
 x_-(-t)=\sup \{ x\in\R,\, \supp y(-t)\in [x(t)-x(-t),+\infty[\} 
 $$
 and $ u\not \equiv 0 $ we must have $ x(-t)+x(t)\ge 0 $ for all $ t\in \R $. 
We claim that this forces 
\begin{equation}\label{aaaa}
\tilde{a}_-=\tilde{a}_+=a_-=a_+ \; .
\end{equation}
Note first that since $ \tilde{a}_-\le \tilde{a}_+ $ and $ a_-\le a_+ $, it suffices to prove that $ \tilde{a}_- \ge a_+ $ and $\tilde{a}_+ \le a_- $. This follows easily by a contradiction argument. Indeed,  assume for instance that $ \tilde{a}_- <a_+$.Then, there exists $ t_0\in \R $ and   $ \varepsilon>0 $ such that $ u(t,x(t)-x_-(-t))<u(t,x(t)+x_+(t)) -\varepsilon $ for all $ t\ge t_0 $.  Since 
 $x(t)-x_-(-t)=q(t-t_0,x(t_0)-x_-(-t_0)) $ and $ x(t)+x_+(t)=q(t-t_0,x(t_0)+x_+(t_0)) $, it follows from \eqref{defq} that 
 $$
 x_+(t)+x_-(-t))\ge \varepsilon (t-t_0) \tendsto{t\to +\infty} +\infty 
 $$
 which contradicts  that $ (x_+(t),x_-(t))\in ]-\infty,r_0]^2 $.  Exactly the same argument but with $ t\to - \infty $ ensures that $\tilde{a}_+ \le a_- $ and completes the proof of the claim \eqref{aaaa}. 

We deduce from \eqref{aaaa} that $ a(t)=a+ $ for all $t\in \R $ and thus \eqref{a'}, \eqref{qq}  and  \eqref{defa} force
$$
u(t,x(0)+x_+(0)+\frac{a_+}{2} t )=\frac{a_+}{2}, \quad \forall t\in \R 
$$
and 
$$
 u_x\Bigl(t,(x(0)+x_+(0)+\frac{a_+}{2}t)-\Bigr)-
u_x\Bigl(t,(x(0)+x_+(0)+\frac{a_+}{2}t)+\Bigr) = a_+, \quad \forall t\in \R \; .
$$
In particular, in view of \eqref{decompositiony},
$$
u(0,x(0)+x_+(0))=\frac{a_+}{2} \mbox{ and } y(0)=a_+ \delta_{x(0)+x_+(0)}+\mu
$$
with $ \mu\in {\mathcal M}_+(\R) $. But this forces $ \mu=0 $ 
since 
$$
 (1-\partial_x^2)^{-1} (a_+  \; \delta_{x(0)+x_+(0)})=\frac{a_+}{2} \exp\Bigl(-|\cdot -(x(0)+x_+(0))|\Bigr)
 $$
 and for any $ \mu\in {\mathcal M}_+(\R) $, with $ \mu\neq 0 $, it holds 
 $$
  (1-\partial_x^2)^{-1} \nu = \frac{1}{2} e^{-|x|} \ast \nu >0 \mbox{ on } \R \; .
 $$
We thus conclude that $ y(0)=a_+ \delta_{x(0)+x_+(0)} $ which leads to  
$$
u(t,x)=\frac{a_+}{2} \exp \Bigl(-\Bigl|x-x(0)-x_+(0)-\frac{a_+}{2} t\Bigr|\Bigr)
$$
\section{Asymptotic stability of the DP peakon}\label{sect4}
We now focus on the case $ b=3 $ that corresponds to the Degasperis-Procesi equation. Recall that in this case 
 \eqref{bfamily} becomes \eqref{DP}.
 The following proposition proven in the appendix ensures that a $Y $- almost  localized global solutions to the DP equation enjoys actually a  uniform exponential decay. 
 \begin{proposition}\label{prodecay}
 Let $ u \in C(\R; L^2(\R)) $  with  $ y=(1-\partial_x^2)u \in C_{w}(\R; \M+) $ be a $Y$-almost localized solution of \eqref{DP} with $ \inf_{\R} \dot{x}\ge c_0>0$.  Then there exists $ C>0 $  such that for all $ t\in\R $, all 
  $ R >0 $  and all $ \Phi\in C(\R) $ with $0\le \Phi\le 1 $ and $ \supp \Phi \subset [-R,R]^c $.
\begin{equation}\label{estimatedecay}
 \int_{\R} \Bigl(4v^2(t) +5v_x^2(t)+v_{xx}^2(t)\Bigr) \Phi(\cdot-x(t)) \, dx +  \dist{ \Phi(\cdot-x(t))}{y(t)} \le C \, \exp( -R/6) \; .
 \end{equation}
 In particular, $ u $ is uniformly in time exponentially decaying.
\end{proposition}
This proposition is a direct consequence  of  an almost monotonicity result for $\H(u)+c_0 M(u) $ at the right of an almost  localized solution that is contained in the following lemma (see the appendix for a sketch of the proof). At this stage it is important to notice that direct calculations lead to 
\begin{equation}\label{equi}
\frac{1}{4} \|w\|_{L^2}^2 \le \H(w) \le \|w\|_{L^2}^2 , \quad \forall u\in L^2(\R) \; .
\end{equation}
As in \cite{MM2}, we introduce the $ C^\infty $-function $ \Psi $ defined  on $ \R $ by 
\begin{equation}\label{defPsi}
\Psi(x) =\frac{2}{\pi} \arctan \Bigl( \exp(x/6)\Bigr) 
\end{equation}
\begin{lemma}\label{almostdecay}
  Let $0<\alpha < 1 $ and let $ u\in C(\R;L^2(\R)) $,  with $ y=(1-\partial_x^2)u \in C_{w}(\R; \M_+) $, be a solution of \eqref{DP} such that there exist $x\,:\, \R\to \R $ of class $ C^1 $ with 
   $ \inf_{\R} \dot{x}\ge c_0>0$  and $ R_0>0 $ with
 \begin{equation}\label{loc}
\|u(t)\|_{L^\infty(|x-x(t)|>R_0)} \le \frac{(1-\alpha) c_0}{2^6} \, , \; \forall t\in\R .
 \end{equation}
  For $ 0<\beta \le \alpha $, $ 0\le \gamma\le \frac{1}{8} (1-\alpha) c_0 $, $ R>0 $, $ t_0\in\R $  and any $ C^1 $-function 
  \begin{equation}\label{condz}
  z\, :\, \R\to \R \mbox{ with } (1-\alpha)  \dot{x}(t) \le \dot{z}(t) \le (1-\beta) \dot{x}(t), \quad \forall t\in \R,
   \end{equation}
  setting
     \begin{equation}\label{defI}
 I^{\mp R}_{t_0} (t)=\dist{5 v^2(t)+4v_x^2(t)+v_{xx}^2(t)+\gamma  y(t)}{\Psi\Bigl(\cdot- z_{t_0}^{\mp R}(t)\Bigr)}
 \end{equation}
  where 
 $$
 z_{t_0}^{\mp R}(t)=x(t_0)\mp R +z(t)-z(t_0)
 $$
 we have 
 \begin{equation}
I^{+R}_{t_0}(t_0)-I^{+R}_{t_0}(t)\le K_0 e^{-R/6} , \quad \forall t\le t_0 \quad  \label{mono}
\end{equation}
and 
 \begin{equation}
I^{-R}_{t_0}(t)-I^{-R}_{t_0}(t_0)\le K_0 e^{-R/6} , \quad \forall t\ge t_0 \quad , \label{mono2}
\end{equation}
for some constant $ K_0>0 $ that only depends on $ \H(u) $, $M(u)$, $c_0$, $R_0$ and $ \beta$. 
\end{lemma}
According to \cite{LL} and \cite{Andre2}, for any speed $ c>0 $ there exists $ \varepsilon_0>0 $ and $ C_0>0 $ such that for any  $ u_0 \in Y_+ $ with 
\begin{equation}\label{stab}
 \|u_0-c\varphi \|_{\H} < \varepsilon^4 \; , \quad 0<\varepsilon<\varepsilon_0 ,
 \end{equation}
 it holds
$$
 \sup_{t\in\R} \|u(t)-c\varphi(\cdot-\xi(t))\|_{\H} <C_0 \varepsilon\; ,
 $$
 where $ u \in C(\R;H^1) $ is the solution emanating from $ u_0$ and  $ \xi(t)\in\R $ is unique point where the function $ 
 v(t,\cdot)=(4-\partial_x^2)^{-1} u(t,\cdot) $ reaches its maximum. According to \cite{Andre1}, by the implicit function theorem, one can prove that there exists $ \varepsilon_0'>0 $ and $ K>1 $ such that if a solution $ u \in C(\R;\H) $ to 
 \eqref{DP} satisfies 
  \begin{equation}\label{gf}
\sup_{t\in\R}  \inf_{y \in\R}  \|u(t)-c\varphi(\cdot-y) \|_{\H} <\varepsilon
\end{equation}
for some $ 0<\varepsilon\le  \varepsilon_0' $ then there exists a uniquely determined  $ C^1$-function $ x\, :\, \R \to \R $ such that 
\begin{equation}\label{fg}
\sup_{t\in\R} \|u(t)-c\varphi(\cdot-x(t))\|_{\H} \le K \varepsilon
\end{equation}
and 
\begin{equation}
\int_{\R} v(t) \rho'(\cdot-x(t))=0, \quad \forall t\in\R \; , \label{ort}
\end{equation}
where $ v=(4-\partial_x)^{-1} u $ and $\rho=(4-\partial_x)^{-1} \varphi $.
Moreover, for all $ t\in\R $, it holds
\begin{equation}\label{estc}
|\dot{x}(t)- c| \le K \varepsilon\; .
\end{equation}
At this stage, we fix $ 0<\theta<c $ and we take 
 \begin{equation}\label{defep}
 \varepsilon= \min \Bigl(\frac{2^{-10}\theta}{ KC_0},\varepsilon_0, \frac{\varepsilon_0'}{C_0}\Bigr)
 \end{equation}
 so that \eqref{stab} ensures that \eqref{fg} and \eqref{estc} hold with $ K\varepsilon \le \frac{\theta}{2^{10}}\le \frac{c}{2^{10}}$. It follows that  $ \dot{x} \ge \frac{3}{4} c $ on $ \R $. Moreover,   combining \eqref{fg}, \eqref{equi}
  and \eqref{dodo} we infer   that there exists $ R>0$ such that 
 $$
 \|u(t, \cdot+x(t))\|_{H^1(]-R,R[^c)} \le 2^{-9} \theta\; .
 $$
 
 Hence, $u $ satisfies the hypotheses of Lemma \ref{almostdecay} for any $ 0<\alpha<1$ such that 
  \begin{equation}\label{okok}
  (1-\alpha)\ge \frac{\theta}{4c} 
  \end{equation}
  and any  $ 0\le \gamma\le\frac{1}{12} (1-\alpha) c $. In particular, $ u$ satisfies the hypotheses of Lemma \ref{almostdecay} for 
   $ \alpha=1/3$. Note that the hypothesis \eqref{difini} with 
 $$
 \eta=\min\Bigl(\frac{2^{-10}\theta}{ KC_0},\varepsilon_0, \frac{\varepsilon_0'}{C_0}\Bigr)^8
 $$
 implies that \eqref{stab} holds with $ \varepsilon $ given by \eqref{defep}.

In the sequel we will make use of the following functionals that measure the quantity $ \H(u)+\gamma M(u) $ at the right and at the left of $ u$.
For $ 0\le \gamma\le \frac{c}{12} $, $ u\in Y$ and $ R>0 $ we set 
 \begin{equation}\label{defJr}
 J_{\gamma,r}^{R}(w)=\dist{5v^2+4v_x^2+v_{xx}^2+\gamma (u-u_{xx})}{\Psi(\cdot -R)} \; .
 \end{equation}
 and
 \begin{equation}\label{defJl}
 J_{\gamma, l}^R(w)=\dist{5v^2+4v_x^2+v_{xx}^2+\gamma (u-u_{xx})}{(1-\Psi(\cdot+R))} 
 \end{equation}
 where $ v=(4-\partial_x^2)^{-1} u $. 
 
 Let $ t_0\in\R $ be fixed.  Fixing $\alpha=\beta=1/3 $ and taking $ z(\cdot)=(1-\alpha) x(\cdot)$, $z(\cdot) $ clearly satisfies \eqref{condz}. Moreover, we have  $ J_{\gamma,r}^{R}(u(t_0, \cdot +x(t_0))=I^{+R}_{t_0}(t_0) $
  where $I^{+R}_{t_0} $ is defined in \eqref{defI}.  Since obviously, 
 $$
 J_{\gamma,r}^R \Bigl(u(t,\cdot+x(t))\Bigr)\ge I^{+R}_{t_0}(t)\; , \quad \forall t\le t_0,
 $$
 we deduce from \eqref{mono} that 
 \begin{equation}\label{monoJr}
 J_{\gamma,r}^R \Bigl(u(t_0,\cdot+x(t_0))\Bigr)\le J_{\gamma,r}^R\Bigl(u(t,\cdot+x(t))\Bigr)+K_0 e^{-R/6} \;  , \quad \forall t\le t_0,
 \end{equation}
 where $ K_0 $ is the constant appearing in \eqref{mono}. 
  Now, let us define 
 \begin{eqnarray*}
 \tilde{I}^{R}_{t_0}(t)  & = & \dist{5v^2(t)+4v_x^2(t)+v_{xx}^2(t)+\frac{c}{12} y(t)}{1-\Psi(\cdot-x(t)+R+\alpha(x(t_0)-x(t)))}\\
 & =& E(u(t))+cM(u(t)) -I^{-R}_{t_0}(t) \; ,
 \end{eqnarray*}
 where we take again $ z(\cdot)=(1-\alpha) x(\cdot)$. 
 Since $ M(\cdot) $ and $ E(\cdot) $ are conservation laws, \eqref{mono2} leads to 
   $$
 \tilde{I}^{R}_{t_0}(t)\ge \tilde{I}^{R}_{t_0}(t_0)-C e^{-R/6} , \; \forall t\ge t_0 \; .
 $$
 We thus deduce as above  that $ \forall t \ge t_0 $, 
 \begin{equation}\label{monoJl}
J_{\gamma,l}^R \Bigl(u(t,\cdot+x(t))\Bigr)\ge J_{\gamma,l}^R\Bigl(x(t_0,\cdot+x(t_0))\Bigr)-K_0 e^{-R/6} \; .
 \end{equation}

The following proposition ensures that, for $\varepsilon$ small enough,  the $ \omega$-limit set for the weak $ H^1$-topology of the orbit of $ u_0 $  is constituted by initial data of  $Y$-almost localized solutions. The crucial tools in the proof are the almost monotonicity properties \eqref{monoJr} and \eqref{monoJl}. We omit the proof since it is exactly the same that the proof of Proposition 5.2  in \cite{L}.
\begin{proposition}\label{propasym}
Let $ u_0 \in Y_+$ satisfying \eqref{stab} with $\varepsilon$ defined as in \eqref{defep}  and let $u \in C(\R;H^1(\R)) $ the emanating solution of \eqref{DP}. For any sequence $ t_n\nearrow +\infty $ there exists a subsequence $ \{t_{n_k}\}\subset \{t_n\} $ and $ \tilde{u}_0\in Y_+ $ such that 
\begin{equation}\label{ppp2}
 u(t_{n_k},\cdot+x(t_{n_k})) \weaktendsto{n_k\to +\infty} \tilde{u}_0 \mbox { in } H^1(\R) 
 \end{equation}
 and 
 \begin{equation}\label{pp2}
 u(t_{n_k},\cdot+x(t_{n_k})) \tendsto{n_k\to +\infty} \tilde{u}_0 \mbox { in } H^1_{loc}(\R) 
 \end{equation}
  where $ x(\cdot) $ is the $ C^1$-function uniquely determined by \eqref{fg}-\eqref{ort}. 
Moreover, the solution of \eqref{DP} emanating from $ \tilde{u}_0 $ is $Y$-almost localized.
\end{proposition}
So, let $ u_0 \in Y_+ $ satisfying  \eqref{stab} with $\varepsilon$ defined as in \eqref{defep} and let  $ t_n\nearrow +\infty $ be a sequence of positive real numbers. According to the above proposition, \eqref{ppp2}-\eqref{pp2} hold for some subsequence $ \{t_{n_k}\}\subset \{t_n\} $ and $ \tilde{u}_0\in Y_+ $ such that 
 the  solution of \eqref{DP} emanating from $ \tilde{u}_0 $ is $Y$-almost localized. Theorem 
\ref{rigidity} then forces 
$$
 \tilde{u}_0= c_0\varphi (\cdot-x_0) 
 $$
 for some $x_0\in\R $ and $ c_0 $ such that $ |c-c_0|\le K \varepsilon\le c/2^9 $. Since \eqref{ppp2}  implies that  
 $$
 v(t_{n_k},\cdot+x(t_{n_k})) \weaktendsto{n_k\to +\infty} \tilde{v}_0 \mbox { in } H^3(\R) 
 $$
 with  $ v_n=(4-\partial_x^2)^{-1} u $ and $  \tilde{v}_0=(4-\partial_x^2)^{-1}  \tilde{u}_0 $, we infer that 
 $\tilde{v}_0 $ satisfies the orthogonality condition \eqref{ort} and thus 
  we must have $ x_0=0 $. On the other hand, \eqref{pp2} and \eqref{fg} ensure that $\displaystyle c_0=\lim_{n\to +\infty} \max_{\R} u(t_{n_k}) $ and thus 
  $$
  u(t_{n_k},\cdot+x(t_{n_k})) -\lambda(t_{n_k})\varphi  \weaktendsto{k\to +\infty} 0 \mbox{ in } H^1(\R)
  $$
where we set $
 \lambda(t):=\max_{\R} u(t) , \quad \forall t\in\R $. Since this is the only possible limit, it follows that 
 $$
  u(t,\cdot+x(t)) -\lambda(t)\varphi  \weaktendsto{t\to  +\infty} 0 \mbox{ in } H^1(\R)\; .
  $$
  and thus 
  \begin{equation}\label{pp3}
 u(t,\cdot+x(t))-\lambda(t)\varphi  \tendsto{t\to 0} 0 \mbox { in } H^1_{loc}(\R) 
 \end{equation}
 and 
  \begin{equation}\label{pp33}
 v(t,\cdot+x(t))-\lambda(t)\rho  \tendsto{t\to 0} 0 \mbox { in } H^3_{loc}(\R) 
 \end{equation}
  \subsection{Convergence in $ H^1(]-A,+\infty[) $ for any $ A>0 $.}\label{51}
 Let $ \delta>0 $ be fixed. Choosing  $ R>0 $ such that $J_{0,r}^R(u(0),\cdot +x(0)) <\delta $ and $K_0 e^{-R/6} \le \delta $,  where $ K_0$ is  the constant that appears in \eqref{monoJr}. We deduce   from \eqref{monoJr}   that 
  $ J_{0,r}^R\Bigl(u(t,\cdot+x(t))\Bigr)<2 \delta $ for all $ t\ge 0 $. This fact together with  the local strong convergence \eqref{pp33} clearly ensure that
  \begin{equation}\label{cvcv}
 v(t,\cdot+x(t))-\lambda(t) \rho  \tendsto{t\to +\infty}  0\mbox{   in } H^2(]-A,+\infty[)  \mbox{ for any } A>0 
 \end{equation}
 and thus 
 $$
  u(t,\cdot+x(t))-\lambda(t) \varphi  \tendsto{t\to +\infty}  0\mbox{   in } L^2(]-A,+\infty[)  \mbox{ for any } A>0 \; .
 $$
 Since $\{u(t), t\in\R\} $ is bounded in $ Y$, this leads to 
  \begin{equation}\label{cvcu}
 u(t,\cdot+x(t))-\lambda(t) \varphi  \tendsto{t\to +\infty}  0\mbox{   in } H^1(]-A,+\infty[)  \mbox{ for any } A>0 \; .
 \end{equation}

 \subsection{Convergence of the scaling parameter}\label{52}
 We claim that 
 \begin{equation}\label{cvlambda}
  \lambda(t)\tendsto{t\to +\infty} c_0\;  .
  \end{equation}
  Let us fix again  $ \delta>0 $ and take  $ R>0 $ such that 
  $ K_0 e^{-R/6} <\delta $. \eqref{monoJl} with $ \gamma=0 $ together with the conservation of $ E(u) $ ensure that, for any couple $ (t,t')\in\R^2$ with $ t>t' $ it holds 
  $$
  \int_{\R} (5v^2+4v_x^2+v_{xx}^2)(t,x) \Psi(x-x(t)+R) \, dx \le  
      \int_{\R} (5v^2+4v_x^2+v_{xx}^2)(t',x) \Psi(x-x(t')+R) \, dx+\delta 
 $$
  On the other hand, by the strong convergence \eqref{cvcv} and the exponential localization of $ \varphi , \varphi' $ and $ \Psi $, there exists $ T>0 $ such that 
   for all $ t\ge T $, 
   $$
    \Bigl| \int_{\R} (5v^2+4v_x^2+v_{xx}^2)(t,x) \Psi(x-x(t)+R) \, dx- \lambda^2(t)E(\varphi) \Bigr| \le \delta \; .
   $$
  It thus follows that 
  $$
  \lambda^2(t) E(\varphi)\le \lambda^2(t') E(\varphi)+3 \delta , \quad \forall t>t'>T \; .
  $$
  Since $ \delta>0 $ is arbitrary, this forces $\lambda $ to have a limit at $ +\infty $ and completes the proof of the claim.
   \subsection{Convergence of $\dot{x} $} \label{53}
   We set  $W(t,\cdot):=c_0\rho( \cdot-x(t))$ and $ \eta(t)=v(t)-c_0 \rho(\cdot-x(t))=v(t)-W(t)$ for all $ t\ge 0 $. Differentiating \eqref{ort} with respect to time and using that $ 4\rho -\rho''=\varphi $, we get 
 $$
\int_{\R} \eta_t  \partial_x W =\dot{ x} \,\int_{\R} 
\eta \partial_x^2 W     = - c_0 \dot{ x} \int_{\R} \eta  \varphi(\cdot -x(t))+ 4 \dot{ x}\int_{\R}   \eta W, \; 
$$
 and thus
\begin{equation}
\Bigl|\int_{\R}  \eta_t  \partial_x W \Bigr|\le   (c_0|\dot{ x}-c_0| +c_0^2 | )
\int_{\R} \eta \varphi(\cdot -x(t)) |+ 4 c_0 | \int_{\R}   \eta W|
 \; . \label{huhu}
\end{equation}
We notice that $ v=(4-\partial_x^2)^{-1} $ is solution of 
\begin{equation}\label{eqv}
v_{t}=-2\partial_{x}v^{2}-\frac{1}{2}\partial_{x}(1-\partial^{2}_{x})^{-1}
(12v^{2}+8v^{2}_{x}+v^{2}_{xx}),~~(t,x)\in\mathbb{R}_{+}\times\mathbb{R}.
\end{equation}
Substituting $ v $ by $ \eta+W$ in \eqref{eqv}  and
using the  equation satisfied by $W$, we obtain the following equation satisfied by $ \eta$ :
  \begin{equation}\label{eqeta}
  \eta_t  - (\dot{x}-c_0)  \partial_x W
= -4 \partial_x  \eta W-(1-\partial_x^2)^{-1}\partial_x \Bigl(8 \eta W +  16\eta_x W_x+ \eta_{xx} W_{xx}\Bigr)\; .
 \end{equation}
 At this stage it is worth noticing that \eqref{cvcv}-\eqref{cvlambda} ensures that 
 \begin{equation}\label{fin}
 |\int_{\R} \eta \varphi(\cdot -x(t))|+| \int_{\R}   \eta W|+\sum_{i=0}^2  \|\partial_x^i \eta(t) \partial_x^i W(t)\|_{L^2} \tendsto{t\to +\infty} 0 \; .
 \end{equation}
Taking the $ L^2 $-scalar product with $ \partial_x W$ with \eqref{eqeta}, integrating by parts, using that $ \|\partial_x W\|_{L^2}^2=
\frac{7}{54} c_0^2$  and the
 decay of $ \rho $ and its first derivative,  \eqref{huhu}, \eqref{fin}, \eqref{fg}  lead to 
 $$
 |\dot{x}(t)-c_0| \tendsto{t\to \infty} 0 \; .
 $$

   \subsection{Strong $ H^1$-convergence on $ ]\theta t ,+\infty[$}\label{54}
   We deduce  from \eqref{cvlambda} that  
 $$
 v(t,\cdot)-c_0 \rho (\cdot-x(t)) \weaktendsto{t\to +\infty}  0\mbox{   in } H^1(\R)  
 $$
 and 
 \begin{equation}\label{jw}
 v(t,\cdot+x(t))-c_0 \rho
 \tendsto{t\to +\infty}  0\mbox{   in } H^1(]-A,+\infty[)  \mbox{ for any } A>0 \; .
 \end{equation}
 \eqref{cvforte} will follow by combining these  convergence results with the almost non increasing property  \eqref{mono}. 
 Indeed,  let us fix $ \delta>0$ and take $ R\gg 1 $ such that 
 \begin{equation}\label{sww}
 \|\rho\|_{H^2(]-\infty,-R/2[}^2< \delta \quad\mbox{and}\quad  \|\Psi-1\|_{L^\infty(]R/2,+\infty[)} <\delta 
\end{equation}
where $ \Psi $ is defined in \eqref{psipsi}. According to the above convergence result there exists $ t_0>0 $ such that 
  $ x(t_0)>R $ and  for all $ t\ge t_0 $, 
 $$
 \int_{-R/2}^{+\infty} (5\eta^2+4\eta_x^2+\eta_{xx}^2)(t,\cdot+x(t))  < \delta \, , 
 $$
 where we set $\eta=v(t)-c_0 \rho(\cdot-x(t)) $. In particular, \eqref{sww} ensures that
 \begin{equation}\label{swww}
\Bigl| E(\varphi)-\int_{\R} \Bigl( 5v(t,\cdot+x(t))\rho +4 v_x(t,\cdot+x(t))\rho_x 
+v_{xx} (t,\cdot+x(t))\rho_{xx}(t,\cdot+x(t)) \Bigr) \Psi(\cdot+y) \Bigr| \lesssim \delta, \quad \forall y\ge R,\, \forall t\ge t_0 \; ,
 \end{equation}
 We set  $ z(t)=\frac{\theta}{2} t$ and notice that  \eqref{okok} ensures that \eqref{condz} is satisfied with $ 1-\alpha=\frac{\theta}{4c} $ and $ \beta=1/4$. Moreover, 
  as noticing in the beginning of this section (see \eqref{okok}), $ u $ satisfies the hypotheses of Lemma \eqref{almostdecay} for such $ \alpha $. 
 According to \eqref{mono2} with $ \gamma=0$, we thus get for all $t\ge t_0$,
 $$
 \int_{\R} (5\eta^2+4\eta_x^2+\eta_{xx}^2)(t,\cdot) \Psi(\cdot-x(t_0)-\frac{\theta}{2}(t-t_0)+R) \le  \int_{\R} (5\eta^2+4\eta_x^2+\eta_{xx}^2)(t_0,\cdot) \Psi(x-x(t_0)+R )+K_0(\alpha)  e^{-R/6}
 $$
  which leads to  
 \begin{align*}
 \int_{\R}&  (5\eta^2+4\eta_x^2+\eta_{xx}^2)(t,\cdot) \Psi\Bigl(\cdot-x(t_0)-\frac{\theta}{2}(t-t_0)+R\Bigr)=
 \int_{\R} (5\eta^2+4\eta_x^2+\eta_{xx}^2)(t,\cdot) \Psi\Bigl(\cdot-x(t_0)-\frac{\theta}{2}(t-t_0)+x_0\Bigr)  \\
 &-2 c_0
 \int_{\R} \Bigl(5 v(t) \rho(\cdot-x(t)) +v_x(t) \rho_x(\cdot-x(t))+v_{xx}(t) \rho_{xx}(\cdot-x(t))\bigr) \Psi\Bigl(\cdot-x(t_0)-\frac{\theta}{2}(t-t_0)+R\Bigr)\\
 &+ c_0^2 \int_{\R} (5\rho^2+4\rho_x^2+\rho_{xx}^2)(t,\cdot-x(t))  \Psi\Bigl(\cdot-x(t_0)-\frac{\theta}{2}(t-t_0)+R\Bigr) \\
 &  \le   \int_{\R} (5\eta^2+4\eta_x^2+\eta_{xx}^2)(t_0,\cdot) \Psi(\cdot-x(t_0)+R )+K_0(\alpha)  e^{-R/6}\\
 &-2 c_0
 \int_{\R} (5 v(t_0) \rho(\cdot-x(t_0)) +4 v_x(t_0) \rho_x(\cdot-x(t_0))+v_{xx}(t_0) \rho_{xx}(\cdot-x(t_0))  \Psi(\cdot-x(t_0)+R )+ C\, \delta \\
 &+ c_0^2 \int_{\R} \Bigl(5\rho^2+4\rho_x^2+\rho_{xx}^2)(t_0,\cdot-x(t_0)) \Bigr) \Psi(\cdot-x(t_0)+R )  +C e^{-R/6}\\
 &  \lesssim  \int_{\R}(5\eta^2+4\eta_x^2+\eta_{xx}^2)(t_0,\cdot)  \Psi(\cdot-x(t_0)+R)+C( e^{-R/6}+\delta) \\
 & \lesssim \delta + e^{-R/6}
 \end{align*}
 where in the next to the  last step we used that $ \rho $ decays exponentially fast and \eqref{swww} since 
  $x(t)-x(t_0)-\frac{\theta}{2}(t-t_0)+R \ge R $ for all $ t\ge t_0$.
  Taking $R$ large enough and $ t_1>t_0$ such that  $ \theta t_1\ge x(t_0)+\frac{\theta}{2} (t_1-t_0)-R $, it follows that for 
  $t\ge t_1 $, 
  $$
 \int_{\R} (5\eta^2+4\eta_x^2+\eta_{xx}^2)(t,\cdot)  \Psi(\cdot-\theta t) \lesssim \delta
 $$
 which completes the proof of the strong $H^2 $ convergence  of $ v(t,\cdot+x(t)) ,$ on $]\theta t,+\infty[ $. The strong $ H^1$-convergence
  of $u(t,\cdot+x(t)) $ follows by using that $ u=(4-\partial_x^2)v $ and that $ u$ is uniformly in time bounded in $ H^{\frac{3}{2}+} $. 
 \subsection{Strong $H^1$-convergence at the left of any given point.}
 In this subsection we complete the proof of Theorem \ref{asympstab} by proving the  lemma below.  As for the {C-H} equation,  the main observation to prove this lemma is that all  the energy of the solutions to the DP equation  that have a non negative density momentum,
   is traveling to the right. This property should be shared by most Hamiltonian CH-type equation  because of the absence  of linear  part.
  \begin{lemma}\label{apen}
   For any $ u_0\in Y_+ $ and any $ z\in\R  $, denoting by $ u\in C(\R;H^1) $  solution of \eqref{DP} emanating from $ u_0$  it holds 
 \begin{equation}\label{tdt}
\lim_{t\to+\infty} \|u(t)\|_{H^1(]-\infty, z[)} = 0 \; .
 \end{equation}
 \end{lemma}
 \begin{proof}
 Let $ 0<\gamma<\|u_0\|_{\H}^2$ and let $ x_\gamma \, :\, \R\to \R $ be defined by 
 \begin{equation}\label{defxg}
 \int_{\R} (5v^2+4v_x^2+v_{xx}^2)(t) \Psi(\cdot -x_\gamma(t)) =\gamma 
 \end{equation}
 with $ \Psi $  defined in \eqref{defPsi}.
 Note that $ x_\gamma(\cdot) $ is well-defined  since $ u_0\in Y_+ $ forces $ u>0 $ on $ \R^2 $ and thus for any fixed $ t\in \R $, 
 $ z\mapsto \int_{\R} (5v^2+4v_x^2+v_{xx}^2)(t) \Psi(\cdot-z) $ is a decreasing continuous bijection from $ \R $ to $ ]0,\|u_0\|_{\H}^2[ $. Moreover, $ u\in C(\R; H^1) $ ensures that $ v\in C(\R; H^3)  $ and thus
  $ x_\gamma(\cdot) $ is a continuous function.  \eqref{tdt} is clearly a direct consequence of the fact that 
  \begin{equation}\label{td3}
 \lim_{t\to +\infty} x_\gamma (t)=+\infty \; .
 \end{equation}
 To prove  \eqref{td3} we first claim that for any $ t\in \R $ and any $ \Delta>0 $ it holds 
 \begin{equation}\label{td4}
 x_\gamma(t+\Delta)-x_\gamma(t) \ge \frac{10}{27}\Bigl(\int_{t}^{t+\Delta} \int_{x_\gamma(t)}^{x_\gamma(t)+2} u^2(\tau,s) \, ds\, d\tau  \Bigr)^{1/2}>0 \; .
 \end{equation}
 Let us prove this claim. First we notice that by continuity with respect to initial data, it suffices to prove \eqref{td4} for $ u\in C^\infty(\R; H^\infty)
 \cap L^\infty(\R; Y_+)  $. Then a simple application of the implicit function theorem ensures that $ t\mapsto x_\gamma (t) $ is of class $ C^1 $. Indeed, 
 $$
 \psi \, :\, (z,v)\mapsto \int_{\R} (5v^2+4v_x^2+v_{xx}^2) \Psi(\cdot-z) 
 $$
 is of class $ C^1 $ from $ \R \times H^2(\R) $ into $ \R $ and for any $ (z_0,v)\in \R\times H^3(\R)/\{0\} $, 
 $\partial_z \psi(z_0,v)=  \int_{\R}  (5v^2+4v_x^2+v_{xx}^2)\Psi'(\cdot-z_0)>0 $.  Now we need the two following Lemmas proved in the appendix :
 \begin{lemma}\label{Lemma 3.2}
Let $u\in C(]-T,T[; H^\infty(\R))$, with $0<T\le +\infty$, be a solution of equation \eqref{DP}. For   any smooth space function $g :\mathbb{R}\mapsto\mathbb{R}$, it holds
\begin{align}
\frac{d}{dt}\int_{\mathbb{R}}&\left(4v^{2}+5v^{2}_{x}+v^{2}_{xx}\right)(t) g\nonumber\\
&=\frac{2}{3}\int_{\mathbb{R}}u^{3}(t) g' +5 \int_{\R} v(t)h(t) g' -4 \int_{\R} v u^2(t) g' +\int_{\R} v_x(t) h_x(t) g' , \quad \forall t\in ]-T,T[ 
\label{4.29}
\end{align}
where  $h=(1-\partial^{2}_{x})^{-1}u^{2}$.
\end{lemma}
 \begin{lemma}\label{Lemma 3.3}
Let $u\in Y_+ $ and $ v=(4-\partial_x^2)^{-1} u $. Then the following estimates hold :
\begin{equation} \label{estuv}
 3 v \le u \le 6 v , \quad |v_x|\le 2 v  \quad \text{ and } |v_{xx}| \le \frac{4}{3} u 
\end{equation}

\end{lemma}

Integrating by parts the last term of the right-hand side member of \eqref{4.29} and using that $h_{xx}  = -u^2 +h $, we infer that 
 \begin{align}
   \dot{x}_\gamma(t) \int_{\R} (5v^2+ & 4v_x^2+v_{xx}^2)\Psi'(\cdot-x_\gamma(t))=\frac{2}{3}\int_{\mathbb{R}}u^{3}(t) \Psi'(\cdot-x_\gamma(t)) \nonumber\\
   & +4 \int_{\R} v(t)h(t) \Psi'(\cdot-x_\gamma(t))-3 \int_{\R} v(t) u^2(t) \Psi'(\cdot-x_\gamma(t))+\int_{\R} v(t) h_x(t) 
   \Psi''(\cdot-x_\gamma(t)) \label{edr}
   \end{align}
 Observe that  by using  integration by parts and $ |u_x|\le u $ we get 
 \begin{align*}
 h(x) & = \frac{1}{2} e^{-x} \int_{-\infty}^x e^\eta u^2(\eta) + \frac{1}{2} e^{x} \int_{x}^{+\infty} e^{-\eta} u^2(\eta)\\
 & =  \frac{1}{2} u^2(x) - e^{-x} \int_{-\infty}^x e^\eta u u_x (\eta)+\frac{1}{2} u^2(x) +e^{x}  \int_x^{+\infty} e^{-\eta} u u_x (\eta)\\
 & \ge  u^2(x) -\int_{\R} e^{-|x-\eta|} u^2 (\eta) \\
  & \ge  u^2(x) - 2h(x)
  \end{align*}
 and thus $ h(x) \ge \frac{1}{3} u^2(x) $.  
  Combining this estimate  with $ |h_x| \le h $ and \eqref{estuv}, using that by direct calculations $
  |\Psi''|\le \Psi'/6 $,  we infer that
 $$
 3 v u^2 \Psi' +v |h_x| |\Psi''|  \le  (3u^2  +\frac{1}{6}v |h_x| ) \Psi' \le (4 v h +\frac{31}{54} u^3)\Psi'
 $$
  Injecting this last  inequality  in \eqref{edr} we eventually get 
 $$
   \dot{x}_\gamma(t) \int_{\R} (5v^2+  4v_x^2+v_{xx}^2)\Psi'(\cdot-x_\gamma(t))
   \ge \frac{5}{54} \int_{\mathbb{R}}u^{3}(t) \Psi'(\cdot-x_\gamma(t)) \; .
  $$
  Noticing that by \eqref{estuv}, $ 5v^2+  4v_x^2+v_{xx}^2\le \frac{37}{9} u^2 $, 
    H\"older inequality together with \eqref{estuv} and  the fact that $\Psi' $ is a non negative function of total mass 1 lead to
 \begin{equation}\label{td6}
 \dot{x}_\gamma(t) \ge \frac{1}{50}\Bigl( \int_{\R} u^2 \Psi'(\cdot-x_\gamma(t))\Bigr)^{1/2} \; .
 \end{equation}
 Integrating this inequality between $ t $ and $ t+\Delta $  yields  \eqref{td4} that obviously implies that $x_\gamma(\cdot) $ is an increasing function. In particular there exists $ x_\gamma^\infty\in \R\cap \{+\infty\} $ such that $x_\gamma(t) \nearrow x_\gamma^\infty $ as $ t\to +\infty $ and it remains to prove that  $x_\gamma^\infty=+\infty $. Assuming the contrary, we first notice that \eqref{estuv} and $ u \le \|y_0\|_{\M} $ on $ \R^2 $ ensure that \eqref{defxg}  leads to
 \begin{equation}\label{contro}
 \lim_{t\to +\infty}  \int_{\R} (5v^2+  4v_x^2+v_{xx}^2)\Psi'(\cdot-x_\gamma(t))= \lim_{t\to +\infty}  \int_{\R} (5v^2+ 4v_x^2+v_{xx}^2)\Psi'(\cdot-x_\gamma^\infty)=\gamma \; .
 \end{equation}
 Now, taking $ \Delta=1 $, \eqref{td4} forces
 $$
 \lim_{t\to  +\infty} \int_{t}^{t+1} \int_{x_\gamma(t)}^{x_\gamma(t)+2} u^2(\tau,s)\, ds\,d\tau =0 
 $$
 which, recalling that $ x_\gamma(t)\to  x_\gamma^\infty$ , leads to 
  $$
 \lim_{t\to  +\infty} \int_{t}^{t+1} \int_{x_\gamma^\infty}^{x_\gamma^\infty+2}u^2(\tau,s) \, ds\, d\tau =0 \; .
 $$
 In particular there exists a sequence $(t_n,x_n)_{n\ge 1}  \subset \R\times  [x_\gamma^\infty,x_\gamma^\infty+2] $ with $ t_n\nearrow +\infty $ such that 
 $ u(t_n,x_n) \to 0 $ as $ n\to \infty $. Therefore,  making use of the fact that $ |u_x|\le u $ on $ \R^2 $ forces, for any $(t,x_0)\in \R^2 $,  that 
 \begin{equation}
 u(t,x) \le e^{|x_0-x|} u(t,x_0) , \quad \forall x\in \R \;  ,
 \end{equation}
 we infer that for any $ A>0 $,
$$
 \lim_{n\to \infty} \sup_{ x\in[x_\gamma^\infty-A,  x_\gamma^\infty+A]} u(t_n,x) =0 \; .
$$
 and \eqref{estuv} then yields 
  \begin{equation}\label{td7}
 \lim_{n\to \infty} \sup_{ x\in[x_\gamma^\infty-A,  x_\gamma^\infty+A]} [5v^2(t_n,x) 
 +4 v_x^2(t_n,x)+v_{xx}^2(t_n,x)] =0 \; .
 \end{equation}
Finally, taking $ A>0 $ such that $ x_\gamma^\infty-A<x_{\gamma'}(0) $ with $\gamma < \gamma'<\|u_0\|_{H^1}^2 $, we infer from \eqref{td7} and the monotonicity of 
 $  t\mapsto x_{\gamma'}(t) $ that 
 $$
 \lim_{n\to \infty}\int_{\R} (5v^2+  4v_x^2+v_{xx}^2)(t_n,\cdot) \Psi(\cdot - x_\gamma^\infty) =\gamma' \; .
 $$
This contradicts \eqref{contro} and concludes the proof of the lemma.
\end{proof}
  \noindent
  \section{Asymptotic stability of train of peakons}\label{6}
In \cite{Andre1} the orbital stability in $L^2(\R) $ of  well ordered trains of peakons is established. More precisely, the following theorem is proved :
\begin{theorem}[\cite{Andre1}]\label{mult-peaks}
Let be given $ N $ velocities $c_1,.., c_N $ such that $0<c_1<c_2<..<c_N $.
There exist   $ A>0 $, $ L_0>0 $
 and $ \varepsilon_0>0 $ such that if  $ u\in C(\R;H^1) $ is
   the global solution of (C-H) emanating from $ u_0\in Y_+ $, with 
 \begin{equation}
 \|u_0-\sum_{j=1}^N \varphi_{c_j}(\cdot-z_j^0) \|_{\H} \le \varepsilon^2 \label{ini}
 \end{equation}
 for some  $ 0<\varepsilon<\varepsilon_0$ and $ z_j^0-z_{j-1}^0\ge L$,
with $ L>L_0 $, then there exist $ N $ $C^1$-functions $t\mapsto x_1(t), ..,t \mapsto x_N(t) $ uniquely determined such that
\begin{equation}
\sup_{t\in\R+} \|u(t,\cdot)-\sum_{j=1}^N \varphi_{c_j}(\cdot-x_j(t)) \|_{\H} \le
A\sqrt{\sqrt{\varepsilon}+L^{-{1/8}}}\;  \label{ini2}
\end{equation}
and 
\begin{equation}
\int_{\R} \Bigl( u(t,\cdot) -\sum_{j=1}^N \varphi_{c_j}(\cdot- x_j(t)) \Bigr)
 \partial_x \varphi_{c_i} (\cdot - x_j(t)) \, dx = 0 \; , \quad i\in\{1,..,N\}. \label{mod2}
\end{equation}
Moreover,   for $ i=1,..,N $ 
\begin{equation}\label{difdif}
|\dot{x}_i-c_i| \le A \sqrt{\sqrt{\varepsilon}+L^{-{1/8}}}, \quad \forall t\in\R_+  \; .
\end{equation}
 \end{theorem}
 This result combined with the asymptotic stability of a single peakon established in the preceding section, yields  the asymptotic stability of a train of well ordered peakons by following the general strategy developped in \cite{MMT} (see also \cite{EM}  ). We do not give the proof but refer the reader to \cite{L} for a detailed proof in the case of the Camassa-Holm equation.
 \begin{theorem}\label{asympt-mult-peaks}
Let be given $ N $ velocities $c_1,.., c_N $ such that $0<c_1<c_2<..<c_N $ and $ 0<\theta_0<c_1/4 $.
There exist   $ L_0>0 $
 and $ \varepsilon_0>0 $ such that if  $ u\in C(\R;H^1) $ is
   the solution of (C-H) emanating from $ u_0\in Y_+ $, with 
 \begin{equation}
 \|u_0-\sum_{j=1}^N \varphi_{c_j}(\cdot-z_j^0) \|_{\H} \le \varepsilon_0^2 
\quad  \mbox{ and } \quad  z_j^0-z_{j-1}^0\ge L_0,\label{inini}
 \end{equation}
then there exist $0< c_1^*<..<c_N^*  $ and $ C^1$-functions $t\mapsto x_1(t), ..,t\mapsto x_N(t) $,  with  $ \dot{x}_j(t) \to c_j^* $ as $ t\to +\infty $, such that,

\begin{equation}\label{mul1}
u(\cdot+x_j(t)) \weaktendsto{t\to +\infty} \varphi_{c_j^*}  \mbox{ in } H^1(\R), \; \forall j\in \{1,..,N\} \; .
\end{equation}

 Moreover, for any $ z\in \R $, 
\begin{equation}\label{mul2}
u-\sum_{j=1}^N \varphi_{c_j^*}(\cdot-x_j(t)) \tendsto{t\to +\infty} 0 \mbox{ in } H^1(
]-\infty,z[ \cup ]\theta_0 t,+\infty[)\; .
\end{equation}
 \end{theorem}
 
    \section{Appendix}
    \subsection{Proof of the Lemma \ref{Lemma 3.2}}

First we notice that applying the operator $(4-\partial_x^2)^{-1} $ to the two members of \eqref{DP} and using that 
\begin{equation}\label{zq}
(4-\partial_x^2)^{-1} (1-\partial_x^2)^{-1}= \frac{1}{3} (1-\partial_x^2)^{-1}-  \frac{1}{3} (4-\partial_x^2)^{-1}\; , 
\end{equation}
we infer that $ v$ satisfies
\begin{equation}\label{vt}
v_t=-\frac{1}{2} h_x \; .
\end{equation}
With this identity in hand it is easy to check that 
\begin{align*}
4 \frac{d}{dt} \int_{\R} v^2 g  &= 8 \int_{\R} v v_t g = -4  \int_{\R} v h_x g \; , 
\end{align*}
and 
\begin{align*}
 5 \frac{d}{dt} \int_{\R} v_x^2 g & = 10 \int_{\R} v_x v_{xt} g 
  = -5 \int_{\R} v_x (1-\partial_x^2)^{-1} \partial_x^2 u^2 g  = 5  \int_{\R} v_x  u^2 g -5 \int_{\R} v_x h g \\
 & = 5  \int_{\R} v_x  u^2 g +5 \int_{\R} v h_x g +5 \int_{\R} v h g'\; 
\end{align*}
In the same way we get 
\begin{align*}
 \frac{d}{dt} \int_{\R} v_{xx}^2 g  &= 2 \int_{\R} v_{xx} v_{xxt} g = -  \int_{\R} 
 v_{xx}  \partial_x (1-\partial_x^2)^{-1} \partial_x^2 u^2 g \\
 & =  \int_{\R} v_{xx} \partial_x (u^2) g-  \int_{\R} 
 v_{xx}   (1-\partial_x^2)^{-1} \partial_x (u^2) g\\
 &= A+B 
\end{align*}
where it holds 
\begin{align*}
 A & =  \int_{\R} (4-\partial_x^2)^{-1} \partial_x^2 u \, \partial_x (u^2) g 
 = - \int_{\R} u  \partial_x (u^2) g+4 \int_{\R} v  \partial_x (u^2) g \\
 & = \frac{2}{3}  \int_{\R} u^3 g' -4   \int_{\R} v_x  u^2 g -4  \int_{\R} v  u^2 g'
 \end{align*}
and 
\begin{align*}
 B& =  \int_{\R} v_x (1-\partial_x^2)^{-1} \partial_x^2( u^2) g
 + \int_{\R} v_x h_x  g' \\
 & = -\int_{\R} v_x  u^2 g  + \int_{\R} v_x h  g + \int_{\R} v_x h_x  g'  \; .
 \end{align*}
 Gathering the above identities, \eqref{4.29} follows.
 \subsection{Proof of Lemma \ref{Lemma 3.3}}
  According to   \eqref{zq} it holds 
  \begin{equation}\label{zqq}
v=(4-\partial_x^2)^{-1} (1-\partial_x^2)^{-1} y =  \frac{1}{3} u - \frac{1}{3} (4-\partial_x^2)^{-1}y
\end{equation}
which proves that $ v\le \frac{1}{3} u $ since $ y\ge 0 $. On the other hand, \eqref{zq} also leads to 
\begin{align*}
6v-u &=  (1-\partial_x^2)^{-1}y - 2 (4-\partial_x^2)^{-1}y \\
&  = \frac{1}{2} e^{-|\cdot|} \ast y - \frac{1}{2} e^{-2|\cdot|} \ast y \\
&= \frac{1}{2} (e^{-|\cdot|}- e^{-2|\cdot|})\ast y \ge 0 
\end{align*}
which proves that $ u \le 6 v $. Now, the identities
$$
v(x)=\frac{e^{-2x}}{4}  \int_{-\infty}^x e^{2 x'} u(x') dx' +\frac{e^{2x}}{4} \int_x^{+\infty} e^{-2x'} u(x') dx'
$$
and 
$$
v_x(x)=-\frac{e^{-2x}}{2}  \int_{-\infty}^x e^{2 x'} u(x') dx' +\frac{e^{2x}}{2} \int_x^{+\infty} e^{-2x'} u(x') dx'\; ,
$$
ensure that $ |v_x| \le 2 v $. Finally, combining the previous estimates  with $ v_{xx}=4v-u $, we eventually get that 
 $ |v_{xx}| \le \frac{4}{3} u $. 
\subsection{Proof of Lemma \ref{almostdecay}}
Let us first notice that  $ \Psi(-\cdot)=1-\Psi $ on $ \R $, $ \Psi' $ is a  positive even  function and that 
 there exists $C>0 $ such that $ \forall x\le 0 $, 
\begin{equation}\label{psipsi}
|\Psi(x)| + |\Psi'(x)|\le C \exp(x/6) \; .
\end{equation}
Moreover, by direct calculations, it is easy to check that 
\begin{equation}\label{psi3}
|\Psi^{'''}| \le  \frac{1}{2} \Psi' \;\text{on }\R 
\end{equation}
and that 
\begin{equation} \label{po}
\Psi'(x)\ge \Psi'(2)= \frac{1}{3\pi} \frac{e^{1/3}}{1+e^{2/3}} , \quad \forall x\in [0,2] \; .
\end{equation}

We first approximate $u(t_0) $ by the sequence of smooth functions $ u_{0,n}=\rho_n\ast u(t_0) $, with $ \{\rho_n\} $ defined in \eqref{rho}, that belongs to $H^\infty(\R) \cap Y_+ $
 and converges to $ u(t_0) $ in $ Y$. According to Propositions \ref{smoothWP} and \ref{WP}, the  sequence of solutions $ \{u_n\} $ to \eqref{DP}  with $ u_n(t_0)=u_{0,n} $  belongs to $ C(\R;H^\infty(\R)) $ and 
 for any fixed $ T>0 $ it holds 
\begin{eqnarray}
u_n & \to & u \mbox{ in } C([t_0-T,t_0+T];H^1) \label{cvH1}\\
v_n & \to & v \text{ in }   C([t_0-T,t_0+T];H^3) \label{cvH3}\\
y_n & \rightharpoonup \! \ast & y  \mbox{ in } C_{ti} (]t_0-T,t_0+T[;{\mathcal M}) \label{weakcv1}
\end{eqnarray}
where $ v_n=(4-\partial_x^2)^{-1} u_n $ and $ y_n=u_n-\partial^2_x u_n $. In particular, for any fixed $ T>0 $, there exists $n_0=n_0(T)\ge 0 $ such that for any $ n\ge n_0 $, 
$$
\|u-u_n\|_{L^\infty(]t_0-T,t_0+T[\times \R)} <  \frac{(1-\alpha) c_0}{2^6} \; ,
$$
which together with   \eqref{loc} force 
\begin{equation}\label{dif}
\sup_{t\in ]t_0-T,t_0+T[}  \|u_n(t)\|_{L^\infty(|x-x(t)|>R_0)} <  \frac{(1-\alpha) c_0}{2^5} \; .
\end{equation}
At this stage it is worth noticing that \eqref{estuv} then ensure that it also holds
\begin{equation}\label{dif2}
\sup_{t\in ]t_0-T,t_0+T[}  \|v_n(t)+|v_x(t)|\|_{L^\infty(|x-x(t)|>R_0)} <  \frac{(1-\alpha) c_0}{2^5} \; .
\end{equation}
We first prove that    \eqref{mono}  holds on $ [t_0-T,t_0] $ with  $ u $ replaced by $ u_n $ for $ n\ge n_0 $. 
The following computations hold for $ u_n $ with $ n\ge n_0$ but , to simplify the notation, we drop the index $ n $.  
For any function $ g\in C^1(\R) $ it is not too hard to check that \eqref{by} with $ b=3 $ leads to 
 \begin{eqnarray}
 \frac{d}{dt}\int_{\R}  y g \, dx & = &-\int_{\R} \partial_x (y u) g -2 \int_{\R} y u_x g \nonumber \\
 & = & \int_{\R} y u g' -2\int_{\R} (u-u_{xx}) u_x g \nonumber \\
  &= &  \int_{\R} y u g' + \int_{\R} (u^2-u_x^2) g'  \label{go2}
 \end{eqnarray}
  Applying \eqref{4.29} and \eqref{go2} with $ g(t,x)=\Psi(x - z^{R}_{t_0}(t)) $ we get
  \begin{eqnarray}
 \frac{d}{dt}I^{+R}_{t_0}(t) & = &-\dot{z}(t) \int_{\R} \Psi' \Bigl[
 4v^{2}+5v^{2}_{x}+v^{2}_{xx}+\gamma y \Bigr] +\gamma \int_{\R} (u^2-u_x^2) \Psi'\nonumber \\
 &  &+ \int_{\mathbb{R}} u \, (\frac{2}{3} u^2 -4u v +\gamma y  )  \Psi' +5 \int_{\R} vh  \Psi' +\int_{\R} v_x h_x  \Psi'
 \nonumber \\
  &\le &   - \dot{z}(t) \int_{\R} \Psi' \Bigl[ 
  4v^{2}+5v^{2}_{x}+v^{2}_{xx} +\gamma y \Bigr] +\gamma \int_{\R} u^2 \Psi'\nonumber \\
   &  &+ \int_{\mathbb{R}} u \, (\frac{2}{3} u^2 -4u v +\gamma y  )  \Psi' + \int_{\R} (5v+|v_x|) h  \Psi' 
 \nonumber \\
  &\le &   - \dot{z}(t) \int_{\R} \Psi' \Bigl[ 
  4v^{2}+5v^{2}_{x}+v^{2}_{xx} +\gamma y \Bigr] +\gamma \int_{\R} u^2 \Psi'+ J_1+J_2 \, \label{go3}
 \end{eqnarray}
 where from the first to the second step we used that $ \Psi'\ge 0 $ and that $ h=(1-\partial_x^2)^{-1} u^2 $ ensures that $ 
 |h_x| \le h $( see  the proof of \eqref{dodo}). \\
 We observe that 
\begin{equation}\label{za}
  \int_{\R} (u^2-u_x^2) \Psi' \le    \int_{\R} u^2  \Psi' = \int_{\R} (4v-v_{xx})^2  \Psi' 
  \le 2 \int_{\R} (16 v^2 +v_{xx}^2) \Psi'\; .
  \end{equation}
  so that, for  $ 0\le \gamma\le \frac{1}{8}(1-\alpha) c_0 $, it holds 
  $$
 -\dot{z}(t) \int_{\R} \Psi' \Bigl[
  4v^{2}+5v^{2}_{x}+v^{2}_{xx}+\gamma y \Bigr] +\gamma \int_{\R} (u^2-u_x^2) \Psi'
   \le  -\frac{\dot{z}(t)}{2} \int_{\R} \Psi' \Bigl[
 4v^{2}+5v^{2}_{x}+v^{2}_{xx}+\gamma y \Bigr] 
 $$
 To estimate $ J_{1} $ we divide $ \R $ into two regions relating to the size of $ |u| $ as follows
\begin{eqnarray}
J_{1}(t) &= & \int_{|x-x(t)|<R_0} u \, (\frac{2}{3} u^2 -4u v +\gamma y  )\Psi'
+ \int_{|x-x(t)|>R_0}  u \, (\frac{2}{3} u^2 -4u v +\gamma y  )\Psi'\nonumber \\
 & = & J_{11}+J_{12}\quad . \label{J0}
\end{eqnarray}
Observe that \eqref{condz} ensures that $ \dot{x}(t)-\dot{z}(t)\ge \beta c_0 $ for all $t\in \R $  and thus, for $ |x-x(t)|<R_0 $,
 \begin{equation}\label{to1}
  x-z_{t_0}^{R}(t)=x-x(t)-R+(x(t)-z(t))-(x(t_0)-z(t_0))\le  R_0-R-\beta c_0 (t_0-t) 
  \end{equation}
  and thus the decay properties of $ \Psi' $ lead  to
\begin{eqnarray}
J_{11} (t) &\lesssim & \Bigl[\|u(t)\|_{L^\infty} (\|u(t)\|_{L^2}^2+\|v\|_{L^2}^2+c_0\|y(t)\|_{L^1})\Bigr]  e^{R_0/6}  e^{-R/ 6}
e^{-\frac{\beta}{6} c_0(t_0-t)} \nonumber \\
 & \lesssim  &  \| u_0\|_{\H}(\|u_0\|_{\H}^2+c_0\|y_0\|_{L^1}) e^{R_0/6} e^{-R/6}
e^{-\frac{\beta}{6} c_0(t_0-t)} \quad . \label{J11}
\end{eqnarray}
where we used that $ \|v\|_{L^2}  \lesssim \|u\|_{\H}$ and that $ u-u_{xx}\ge 0 $ ensures that 
$$
 \|u\|_{L^\infty}^2 \le \|u\|_{H^1}^2 \le 2  \|u\|_{L^2}^2 \lesssim \|u\|_{\H}^2 \; .
 $$
On the other hand,  \eqref{dif} ,Young's inequality  and \eqref{za} lead  for all $ t\in [t_0-T,t_0] $ to
\begin{eqnarray}
J_{12} &\le & 4 \| u\|_{L^\infty(|x-x(t)|>R_0)} \int_{|x-x(t)|>R_0} (u^2+v^2 +\gamma y)\Psi'\nonumber \\
 & \le & \frac{ (1-\alpha) c_0 }{8}  \int_{|x-x(t)|>R_0} 
 \Bigr[  4v^{2}+5v^{2}_{x}+v^{2}_{xx} +\gamma y \Bigr]\Psi' \quad .\label{J12}
\end{eqnarray}
 It thus remains to estimate
 $ J_2(t) $.
For this,  we decompose again $
  \R $ into two regions relating to the size of $ \max(v, |v_x|) $.
   First proceeding as
    in \eqref{J11} we easily check that 
  \begin{eqnarray}
& & \int_{|x-x(t)|<R_0}(5v+|v_x|) \Psi'
(1-\partial_x^2)^{-1}(u^2) \nonumber \\
& & \le \frac{5}{2} \|v+|v_x|\|_{L^\infty} \sup_{|x-x(t)|<R_0}
|\Psi'(x- z^{R}_{t_0}(t))|\int_{\R} e^{-|x|} \ast u^2  \, dx \nonumber \\
 &  & \le C \|u_0\|_{\H}^3 \, e^{R_0/6} e^{-R/6}
 e^{-\frac{\beta}{6}c_0(t-t_0)} \label{J31}
\end{eqnarray}
since $  \|v+|v_x|\|_{L^\infty} \lesssim \|v\|_{H^2} \lesssim \|u\|_{\H} $ and 
\begin{equation}
 \forall f\in L^1(\R), \quad (1-\partial_x^2)^{-1} f =\frac{1}{2} e^{-|x|} \ast f \quad .
 \label{tytu}
 \end{equation}
Now in the region $ |x-x(t)|>R_0$, noticing that $ \Psi' $ and
$ u^2 $ are non-negative, we  get
 \begin{eqnarray}
 & &  \int_{|x-x(t)|>R_0}(5v+|v_x|)\Psi'
(1-\partial_x^2)^{-1}(u^2) \nonumber \\
 &  \le & 5 
\|v(t)+|v_x(t)|\|_{L^\infty(|x-x(t)|>R_0)}\int_{|x-x(t)|>R_0}\Psi'(
(1-\partial_x^2)^{-1}(u^2) \nonumber \\
&  \le &  5  \| v(t)+|v_x(t)|\|_{L^\infty(|x-x(t)|>R_0)}  \int_{\R} (u^2) (1-\partial_x^2)^{-1}
\Psi'
\end{eqnarray}
On the other hand, from
  \eqref{psi3} and \eqref{tytu}   we  infer  that 
  $$
(1-\partial_x^2) \Psi' \ge \frac{1}{2} \Psi' \Rightarrow
(1-\partial_x^2)^{-1} \Psi'\le 2  \Psi' \; .
  $$
Therefore, on account of \eqref{dif2} and \eqref{za},
\begin{eqnarray}
 & &  \int_{|x-x(t)|>R_0} (5v+|v_x|)  \Psi'
(1-\partial_x^2)^{-1}(u^2) \nonumber \\
&  &  \le 10 \| v(t)+|v_x(t)|\|_{L^\infty(|x-x(t)|>R_0)}  \int_{\R} u^2
 \Psi' \nonumber \\
 &  &  \le   \frac{ (1-\alpha)  c_0}{8}
\int_{\R} ( 4v^{2}+5v^{2}_{x}+v^{2}_{xx})
 \Psi' \label{J32}
\end{eqnarray}
Gathering \eqref{J0}, \eqref{J11}, \eqref{J12}, \eqref{J31} and
\eqref{J32} we conclude that there exists  $C $  only
depending on  $R_0 $, M(u)  and $ \H(u) $ 
  such
that for  $ R \ge R_0 $ and $ t\in [-T+t_0,t_0] $ it holds 
\begin{equation}
\frac{d}{dt} I^{+R}_{t_0}(t) \le  C  e^{-R/6} e^{-\frac{\beta}{6}(t_0-t)} \; .
\label{nini}
\end{equation}
Integrating between $ t$ and $ t_0$  we obtain \eqref{mono} for  any $  t \in [ t_0-T,t_0] $ and $ u $ replaced by $ u_n $ with $ n\ge n_0$. Note that the constant appearing in front of the exponential now also depends on $ \beta$. 
 The convergence results \eqref{cvH1}-\eqref{weakcv1} then ensure that \eqref{mono} holds also for  $ u $ and   $t\in[ t_0-T,t_0] $ and  the result for $ t\le t_0 $ follows since $ T>0 $ is arbitrary. Finally, \eqref{mono2} can be proven in exactly the same  way by noticing that  for  $|x-x(t)|<R_0 $ it holds 
  \begin{equation} \label{to2}
   x-z_{t_0}^{-R}(t)=x-x(t)+R+(x(t)-z(t))- (x(t_0)-z(t_0))\ge  -R_0+R+\beta c_0 (t-t_0) \; .
  \end{equation}
\hfill $\square$ \\
\noindent
\subsection{Proof of Proposition \ref{prodecay}}
First, as explained in Remark \ref{remark1}, the    $Y$-almost localization of $ u $ implies that $ u $ is $ H^1$-almost localized and since $v=(4-\partial_x^2)^{-1} u$ the same type arguments show that $ v$ is $ H^3$-almost localized.  Therefore, it is clear that $ u $ satisfies the hypotheses of Lemma \ref{almostdecay} for $\alpha=1/3$ and $ R_0 >0 $ big enough. 
We  fix $\alpha=1/3 $ and take $ \beta=1/3 $, $ \gamma=\frac{c_0}{12}$ and $ z(\cdot)=\frac{2}{3} x(\cdot)$ which clearly satisfy \eqref{condz}. Let us   show that $ I^{+R}_{t_0}(t) \tendsto{t\to-\infty} 0 $ which
together with \eqref{mono}
 will clearly lead to
 \begin{equation}
I^{+R}_{t_0}(t_0) \le  C e^{-R/6} \label{gigi}\quad .
 \end{equation}
For $ R_\varepsilon>0 $ to be specified later we decompose
  $ I^{+R}_{t_0} $ into
  \begin{eqnarray*}
I^{+R}_{t_0}(t)& = &\dist{4v^{2}(t)+5v^{2}_{x}(t)+v^{2}_{xx}(t)+\frac{c_0}{12}  y(t)}{\Psi(\cdot-z^{R}_{t_0}(t))\Bigl(1-\phi(\frac{\cdot-x(t)}{R_\varepsilon})\Bigr)}\\
& & +\dist{4v^{2}(t)+5v^{2}_{x}(t)+v^{2}_{xx}(t)+\frac{c_0}{12} y(t)}{\Psi(\cdot-z^R_{t_0}(t))
\phi(\frac{\cdot-x(t)}{R_\varepsilon})}
\\
 &= & I_1(t)+I_2(t) \quad .
  \end{eqnarray*}
  where $ \phi\in C^\infty(\R) $ is supported in $[-1,1] $ with $ 0\le \phi\le 1 $ on $ [-1,1] $ and $ \phi\equiv 1 $ on $[-1/2,1/2]$. 
From the $Y$-almost localization of $ u $ and the $ H^2(\R) $-almost localization of $ v$, for any $ \varepsilon>0
$ there exists $ R_\varepsilon>0 $
 such that $ I_1(t) \le \varepsilon/2 $. On the other hand, we observe that
 $$
 I_2(t) \le  (\|u_0\|_{\H}^2+c_0\|y_0\|_{\mathcal{M}}) \Psi\Bigl(R_\varepsilon-R-\frac{1}{3}(x(t_0)-x(t))\Bigr) \quad.
 $$
  But
  $ \dot{x}>c_0>0 $  obviously  imply that, for $ |x-x(t)|\le R_\varepsilon$, 
  $$x-z^{+R}_{t_0}(t)=x-x(t)-R -\frac{1}{3} (x(t_0)-x(t))  \le R_\varepsilon-R-\frac{1}{3} c_0 (t_0-t) \tendsto{t\to -\infty} -\infty
  $$
  which
  proves our claim since $\displaystyle\lim_{x\to -\infty} \Psi(x) =0 $. \\
 It follows from  \eqref{gigi} that for all $ t\in\R $, all 
  $x_0>0 $  and all $ \Phi\in C(\R) $ with $0\le \Phi\le 1 $ and $ \supp \Phi \subset [x_0, +\infty[ $.
$$
 \int_{\R} (4v^{2}(t)+5v^{2}_{x}(t)+v^{2}_{xx}(t))  \Phi(\cdot-x(t)) \, dx + \frac{c_0}{12} \dist{ \Phi(\cdot-x(t))}{y(t)} \le C \, \exp( -x_0/6) \; .
 $$
  The invariance of (C-H) under the transformation $ (t,x) \mapsto (-t,-x) $ yields the result for $  \supp \Phi \subset ]-\infty,-x_0] $.  Finally, the identity  $ u=(4-\partial_x^2) v $ together with \eqref{dodo}   ensure that 
  $$
   \int_{\R} (u^2(t)+u_x^2(t))  \Phi(\cdot-x(t)) \, dx  \le C' \, \exp( -x_0/6) \; .
  $$
  and the Sobolev embedding $ H^1(\R)  \hookrightarrow L^\infty(\R) $ enables to conclude that $ u $ is uniformly exponentially decaying.
    
 \noindent
  {\bf Acknowledgements} 
  The author thank ....
  \vspace{3mm}\\
\noindent
 {\bf Conflict of Interest }: The author declares that he has no conflict of interest.

\end{document}